\documentclass{article}%
\usepackage{amsmath}
\usepackage{amsfonts}
\usepackage{amssymb}
\usepackage{graphicx}
\usepackage{color}%
\usepackage{mathrsfs}
\usepackage{mathscinet}
\usepackage{mathtools}
\usepackage{nameref}
\usepackage[title,titletoc]{appendix}
\usepackage{titlesec}
\titlelabel{\thetitle.\ }
\providecommand{\U}[1]{\protect\rule{.1in}{.1in}}
\newtheorem{theorem}{Theorem}
\newtheorem{acknowledgement}[theorem]{Acknowledgement}

\newtheorem{definition}[theorem]{Definition}

\newtheorem{lemma}[theorem]{Lemma}

\newtheorem{remark}[theorem]{Remark}

\newenvironment{proof}[1][Proof]{\noindent\textbf{#1} }{\ \rule{0.5em}{0.5em}}
\numberwithin{theorem}{section}

\begin{document}
 
\author{}
\title{\scshape Calculus of Variations: A Differential Form Approach }
\date{}
\maketitle

\centerline{\scshape Swarnendu Sil}
\medskip
{\footnotesize
 \centerline{ Section de Math\'{e}matiques}
   \centerline{Station 8, EPFL}
   \centerline{1015 Lausanne, Switzerland}
   \centerline{swarnendu.sil@epfl.ch}
}

\begin{abstract}
 We prove existence and up to the boundary regularity estimates in $L^{p}$ and H\"{o}lder spaces for weak solutions of the linear system 
\begin{equation*}
 \delta \left( A d\omega \right)  + B^{T}d\delta \left( B\omega \right)  =  \lambda B\omega + f  \text{ in } \Omega,
\end{equation*}
with either $ \nu\wedge \omega$ and $\nu\wedge \delta \left( B\omega \right)$ or $\nu\lrcorner B\omega$ and 
$\nu\lrcorner \left( A d\omega \right)$ prescribed on $\partial\Omega.$ The proofs are in the spirit of  `Campanato method' and thus 
avoid potential theory and do not require a verification of Agmon-Douglis-Nirenberg or Lopatinski\u{i}-Shapiro type
 conditions. Applications to a number of related problems, such as general versions of the time-harmonic Maxwell system, stationary Stokes problem and the `div-curl' systems, are included. 
\end{abstract}
 \smallskip


\section{Introduction}
Up to the boundary regularity results for second order linear elliptic systems in divergence form with Dirichlet or conormal derivative type boundary condition are well-known. However, a large class of elliptic systems can be written 
more crisply in the language of differential forms and often, for such systems, neither of those is the relevant boundary condition. A 
typical example is the Poisson problem for the Hodge Laplacian with prescribed `tangential part' or prescribed `normal part' 
on the boundary respectively, namely the systems,
\begin{equation*}
 \left\lbrace \begin{gathered}
                \delta d\omega  + d\delta \omega  =  f  \text{ in } \Omega, \\
                \nu\wedge \omega = 0 \text{  on } \partial\Omega. \\
                \nu\wedge \delta \omega = 0 \text{ on } \partial\Omega. 
                \end{gathered} 
                \right. \qquad \text{ or } \qquad \left\lbrace \begin{gathered}
                \delta d\omega  + d\delta \omega  =  f  \text{ in } \Omega, \\
                \nu\lrcorner \omega = 0 \text{  on } \partial\Omega. \\
                \nu\lrcorner d\omega = 0 \text{ on } \partial\Omega. 
                \end{gathered}\right.
\end{equation*}
The regularity results for these two systems are known since Morrey\cite{MorreyHarmonic2} (see also \cite{Morrey1966}). However, regularity results for more general linear elliptic systems with 
these type of boundary conditions do not exist in the literature. The reason for this surprising absence probably lies in the available proofs of these results. The original proof of Morrey 
relied on potential theory and used the fact that $\delta d + d \delta$, i.e the Hodge 
Laplacian is precisely the componentwise scalar Laplacian. Also, after flattening the boundary, the condition $\nu\wedge\omega = 0$ and $\nu \wedge \delta\omega = 0$ imply 
that as far as the principal order terms are concerned, the whole system decouples and gets reduced to $\binom{n}{k}$ number of scalar Poisson problems with lower order terms, out of which 
$\binom{n-1}{k}$ number of equations has zero Dirichlet boundary condition and the other $\binom{n-1}{k-1}$ number 
of equations has zero Neumann boundary condition. Other available proofs verify either the Lopatinski\u{i}-Shapiro, henceforth LS, (see Schwarz\cite{SchwarzHodge}) or 
 the Agmon-Douglis-Nirenberg complementing condition, henceforth ADN, (see Csato \cite{Csatothesis}) and these verifications too rely on the fact that the principal symbol 
 of the operator is rather `simple'. 
 
 \paragraph*{} Deriving the regularity results for the system $\delta \left( A d\omega \right)  + B^{T}d\delta \left( B\omega \right)  =   f $ or even the simpler system 
 $\delta \left( A d\omega \right)  + d\delta \omega  =   f$ with these type of boundary conditions calls for different methods, as the verification of the ADN or LS conditions 
 for these systems looks algebraically tedious. On the other hand, the boundary conditions simply arise out of an integration by parts formula and in principle, verifying ADN or LS should be an 
 avoidable overkill. 
 
\paragraph*{} For linear elliptic systems with Dirichlet boundary conditions, the classical Campanato method (see Campanato\cite{CampanatoEllipticsystem}, 
also Giaquinta-Martinazzi\cite{giaquinta-martinazzi-regularity} and references therein) of deriving estimates in Morrey and Campanato spaces, 
avoids potential theory and yields at the same time both the Schauder estimates and via an interpolation theorem of Stampacchia\cite{StampacchiaInterpolation}, also 
 the $L^p$ estimates. This approach has also been adapted to elliptic systems with conormal derivative type condition, for example in Giaquinta-Modica\cite{GiaquintaModicaStokes}. 
 The crux of the present article is to adapt this approach to these kind of boundary conditions. This, as a particular case  yields a new proof for the regularity of the Hodge 
 Laplacian system as well.  
 
\paragraph*{} The main result of the present article (Theorem \ref{generalHodgesystemtheorem} and \ref{generalHodgesystemtheoremnormal}) is the existence and up to the boundary regularity results in H\"{o}lder and $L^{p}$ spaces for the system 
\begin{equation}\label{introhodgeelliptic}
 \left\lbrace \begin{gathered}
                \delta ( A d\omega )  + B^{T}d \delta\left( B\omega \right)   =   f  \quad \text{ in } \Omega, \\
                \nu\wedge \omega = \nu\wedge\omega_{0} \quad \text{  on } \partial\Omega, \\
                \nu\wedge \delta \left( B\omega \right) = \nu\wedge\delta \left( B \omega_{0} \right) \quad \text{ on } \partial\Omega, 
                \end{gathered} 
                \right. 
\end{equation}
and also its counterpart with the `normal condition' 
\begin{equation}\label{introhodgeellipticnormal}
 \left\lbrace \begin{gathered}
                \delta ( A d\omega )  + B^{T} d \delta\left( B \omega \right)   =   f  \text{ in } \Omega, \\
                \nu\lrcorner \left( B \omega \right) = \nu\lrcorner\left( B \omega_{0} \right) \text{  on } \partial\Omega. \\
                \nu\lrcorner \left( A d\omega \right) = \nu\lrcorner \left( A d\omega_{0} \right)  \text{ on } \partial\Omega, 
                \end{gathered} 
                \right. 
\end{equation}
where $A$ and $B$ are matrix fields and $\omega$ is a $k$-form. Note that unlike the case of the Hodge Laplacian, where solving the tangential boundary value problem for 
$k$-forms is equivalent to solving the normal boundary value problem for $(n-k)$-forms by Hodge duality, these two problems are not dual to each other in general. Instead 
each has their respective dual versions. 
\paragraph*{} The results for \eqref{introhodgeelliptic} and \eqref{introhodgeellipticnormal} yield also the existence and regularity results for a number of related problems. 
The time-harmonic Maxwell's equation in a bounded domain in $\mathbb{R}^{3}$ is 
\begin{align*}
  \left\lbrace \begin{aligned}
                \operatorname*{curl}  H  &=  i\omega \varepsilon E + J_{e}    
                &&\text{ in } \Omega, \\
                \operatorname*{curl} E &= -i\omega \mu H + J_{m}     &&\text{ in } \Omega, \\
                \nu \times E &= \nu \times E_{0} &&\text{  on } \partial\Omega.
                \end{aligned} 
                \right.
\end{align*}
Eliminating $H$ and writing as a second order system in $E$, we obtain, 
\begin{align*}
 \left\lbrace \begin{aligned}
                \operatorname*{curl} ( \mu^{-1} \operatorname*{curl} E  ) &=  \omega^2 \varepsilon E -i\omega J_{e} + \operatorname*{curl}\left( \mu^{-1} J_{m}\right)    
                &&\text{ in } \Omega, \\
                \operatorname*{div} ( \varepsilon E ) &= \frac{i}{\omega}\operatorname*{div} J_{e} &&\text{ in } \Omega, \\
                \nu \times E &= \nu \times E_{0} &&\text{  on } \partial\Omega.
                \end{aligned} 
                \right. 
\end{align*}
$H$ also satisfies a similar second order system with normal boundary conditions. Writing in the language of differential form, the systems for $E$ and $H$ are 
special cases of the general systems 
\begin{equation*}\label{introgeneralmaxwell} 
   \left\lbrace \begin{aligned}
                \delta ( A du )   &= f  &&\text{ in } \Omega, \\
                \delta \left( B u \right) &= g &&\text{ in } \Omega, \\
                \nu\wedge u &= \nu\wedge u_{0} &&\text{  on } \partial\Omega, 
                \end{aligned} 
                \right. \quad \text{ and } \quad \left\lbrace \begin{aligned}
                \delta ( A du )   &=  f  &&\text{ in } \Omega, \\
                \delta \left( B u \right) &= g &&\text{ in } \Omega, \\
                \nu\lrcorner Bu &= \nu\lrcorner Bu_{0} &&\text{  on } \partial\Omega, \\
                \nu\lrcorner \left( A du \right) &= \nu\lrcorner \left( A du_{0} \right) &&\text{  on } \partial\Omega,
                \end{aligned} 
                \right. 
\end{equation*}
respectively, for $k$-forms $u$. When $B$ is the identity matrix, these two systems are respectively related to general versions of 
the vorticity-velocity-pressure formulation of Stationary Stokes system with two different types of boundary conditions, both of which are studied
 in the context of Stokes flow in three dimensions (see \cite{BeiraoBerselliNavierStokesstressfreebc}, \cite{ConcaPironneauNavierStokes}, \cite{DuboisStokes}). 
\paragraph*{} On our way to prove the above results, we also show existence and apriori regularity estimates for the generalized `div-curl' systems 
\begin{equation*}
   \left\lbrace \begin{aligned}
                d(A(x)u) &= f  &&\text{ in } \Omega,\\
                \delta (B(x) u) &= g &&\text{ in } \Omega, \\
                \nu\wedge A(x)u &= \nu\wedge u_0 &&\text{  on } \partial\Omega,
                \end{aligned} 
                \right. \quad \text{ and } \quad \left\lbrace \begin{aligned}
                d(A(x)u) &= f  &&\text{ in } \Omega, \\
                \delta (B(x) u) &= g &&\text{ in } \Omega, \\
                \nu\lrcorner B(x)u &= \nu\lrcorner u_0 &&\text{  on } \partial\Omega.
                \end{aligned} 
                \right. \end{equation*}
 These systems have been studied in various degrees of generality in connection to Gaffney-Friedrichs type inequalities, starting with the work of 
 Friedrichs\cite{FriedrichsGaffney}, Gaffney\cite{GaffneyHarmonicoperator},\cite{GaffneyHarmonicintegrals} and Morrey\cite{MorreyHarmonic2} 
 (see also \cite{CsatoDacKneuss},\cite{DacGanboKneussSymplectic}, \cite{SaranenDivcurl}). We also prove the existence of a solution with optimal regularity for the Dirichlet boundary value problem 
\begin{equation*}
 \left\lbrace \begin{gathered}
                \delta ( A (x) du )  = f   \text{ in } \Omega, \\
                u = u_{0}  \text{  on } \partial\Omega,
                \end{gathered} 
                \right. 
\end{equation*} 
for $k$-forms $u.$ Note that when $k=0$, the problem is just the Dirichlet problem for the elliptic equation $$\operatorname*{div}\left( A \nabla u \right) = f $$  whereas 
as soon as $k \geq 1,$ it becomes a system which is not elliptic and has an infinite dimensional space of solutions. 

\paragraph*{} Some of these results, especially the ones concerning time-harmonic Maxwell system have been studied extensively in the past. Some results 
are already known in the special cases when $u$ is a vector 
field (corresponding to $k=1$ in our case) and $n=3$ (see e.g. 
\cite{AlbertiCapdeboscqMaxwell}, \cite{leisMaxwellanisotropic}, \cite{weber-regularitymaxwell}). But those proofs are based 
on scalar elliptic theory componentwise, using the fact that for vector fields in three dimensions, both the boundary conditions and the systems are simple enough to find explicitly, by direct calculation, the equations satisfied by 
each component. Some related results in the general case of $k$-forms can be found 
in \cite{PicardMaxwellcompactembedding},\cite{WeckMaxwell}. The advantage of our approach is that it brings out the common core of the structural features of all these related systems that is ultimately responsible for regularity.

\section{Notations}
We now fix the notations, for further details we refer to
\cite{CsatoDacKneuss}. Let $n \geq 2$ and $0 \leq k \leq n$ be an integer. \begin{itemize}
\item We write $\Lambda^{k}\left(  \mathbb{R}^{n}\right)  $ (or simply
$\Lambda^{k}$) to denote the vector space of all alternating $k-$linear maps
$f:\underbrace{\mathbb{R}^{n}\times\cdots\times\mathbb{R}^{n}}_{k-\text{times}%
}\rightarrow\mathbb{R}.$ For $k=0,$ we set $\Lambda^{0}\left(  \mathbb{R}%
^{n}\right)  =\mathbb{R}.$ Note that $\Lambda^{k}\left(  \mathbb{R}%
^{n}\right)  =\{0\}$ for $k>n$ and, for $k\leq n,$ $\operatorname{dim}\left(
\Lambda^{k}\left(  \mathbb{R}^{n}\right)  \right)  ={\binom{{n}}{{k}}}.$

\item $\wedge,$ $\lrcorner\,,$ $\left\langle \ ;\ \right\rangle $ and,
respectively, $\ast$ denote the exterior product, the interior product, the
scalar product and, respectively, the Hodge star operator. However, $\ast$ in the superscript denotes pullback operation. If $T:\mathbb{R}^{n} \rightarrow \mathbb{R}^{n}$ is 
a linear map and $\xi \in \Lambda^{k},$ then $T^{\ast}\xi \in \Lambda^{k}$ stands for the usual pullback of $\xi.$ For pullback of differential forms, see later. 

\item If $\left\{  e^{1},\cdots,e^{n}\right\}  $ is a basis of $\mathbb{R}%
^{n},$ then, identifying $\Lambda^{1}$ with $\mathbb{R}^{n},$%
\[
\left\{  e^{i_{1}}\wedge\cdots\wedge e^{i_{k}}:1\leq i_{1}<\cdots<i_{k}\leq
n\right\}
\]
is a basis of $\Lambda^{k}.$ An element $\xi\in\Lambda^{k}\left(
\mathbb{R}^{n}\right)  $ will therefore be written as%
\[
\xi=\sum_{1\leq i_{1}<\cdots<i_{k}\leq n}\xi_{i_{1}i_{2}\cdots i_{k}%
}\,e^{i_{1}}\wedge\cdots\wedge e^{i_{k}}=\sum_{I\in\mathcal{T}^{k}}\xi
_{I}\,e^{I}%
\]
where%
\[
\mathcal{T}^{k}=\left\{  I=\left(  i_{1}\,,\cdots,i_{k}\right)
\in\mathbb{N}^{k}:1\leq i_{1}<\cdots<i_{k}\leq n\right\}  .
\]
We shall identify exterior $1$-forms with vectors freely and shall refrain from using the musical notation to denote these identifications, in order not to 
burden our notations further. Also, we shall often write an exterior $k$-form as a vector in $\mathbb{R}^{\tbinom{n}{k}}$, when the alternating structure is not important for our concern.
In a similar vein, we shall identify $m \times n$ matrices with the space $\mathbb{R}^{m \times n}.$ Also, given any square matrix $X,$ $X^{T}$ and $X^{-1}$ denotes its transpose and its inverse, 
if it exists, respectively. We shall employ $\mathbf{I}$ to denote the identity matrix, whose size would be clear from the context.  
\end{itemize}

\noindent Let $0\leqslant k\leqslant n$ and let $\Omega\subset\mathbb{R}^n$ be open, bounded and smooth. 
\begin{itemize}
\item A differential $k$-form $\omega$ is a
measurable function $\omega:\Omega\rightarrow\Lambda^{k}.$
\item The usual Lebesgue, Sobolev and H\"{o}lder spaces are defined componentwise and are denoted by their usual symbols. 
\item If $\omega \in L^{1}\left( \Omega ; \Lambda^{k}\right)$ and $\Phi: \Omega^{'}\subset \mathbb{R}^{n}\rightarrow \Omega$ is an orientation preserving diffeomorphism,
then $\Phi^{\ast}\left( \omega \right) \in L^{1}\left( \Omega^{'}; \Lambda^{k}\right)$ stands for the pullback of $\omega$.

\item For $1\leqslant p <  \infty$ and $\lambda \geq 0,$  $\mathrm{L}^{p,\lambda}\left(\Omega;\Lambda^{k}\right) $ stands for the Morrey space of all $\omega \in L^{p}\left(\Omega;\Lambda^{k}\right)$ such that 
$$ \lVert \omega \rVert_{\mathrm{L}^{p,\lambda}\left(\Omega;\Lambda^{k}\right)}^{p} := \sup_{\substack{ x_{0} \in \overline{\Omega},\\ \rho >0 }} 
\rho^{-\lambda} \int_{B_{\rho}(x_{0}) \cap \Omega} \lvert \omega \rvert^{p} < \infty, $$ endowed with the norm 
$ \lVert \omega \rVert_{\mathrm{L}^{p,\lambda}\left(\Omega;\Lambda^{k}\right)}$ and $\mathcal{L}^{p,\lambda}\left(\Omega;\Lambda^{k}\right) $ denotes the Campanato space of all $\omega \in L^{p}\left(\Omega;\Lambda^{k}\right)$ such that 
$$ [\omega ]_{\mathcal{L}^{p,\lambda}\left(\Omega;\Lambda^{k}\right)}^{p} := \sup_{\substack{ x_{0} \in \overline{\Omega},\\  \rho >0 }} 
\rho^{-\lambda} \int_{B_{\rho}(x_{0}) \cap \Omega} \lvert \omega  - (\omega)_{ \rho , x_{0}}\rvert^{p} < \infty, $$ where 
$$\displaystyle (\omega)_{ \rho , x_{0}} = \frac{1}{\operatorname*{meas} \left( B_{\rho}(x_{0}) \cap \Omega \right)}\int_{B_{\rho}(x_{0}) \cap \Omega} \omega ,$$  
endowed with the norm 
$ \lVert \omega \rVert_{\mathcal{L}^{p,\lambda}\left(\Omega;\Lambda^{k}\right)} := \lVert  \omega \rVert_{L^{p}(\Omega, \Lambda^{k})} +  
[\omega ]_{\mathcal{L}^{p,\lambda}\left(\Omega;\Lambda^{k}\right)}.$ For standard facts about these spaces, particularly their identification with H\"{o}lder spaces and BMO space, see \cite{giaquinta-martinazzi-regularity}. 
\item Two particular differential operators on differential forms will have a special significance for us. A differential
$(k+1)$-form $\varphi\in L^{1}_{\rm loc}(\Omega;\Lambda^{k+1})$
is called the exterior derivative of $\omega\in
L^{1}_{\rm loc}\left(\Omega;\Lambda^{k}\right),$ denoted by $d\omega$,  if
$$
\int_{\Omega} \eta\wedge\varphi=(-1)^{n-k}\int_{\Omega} d\eta\wedge\omega,
$$
for all $\eta\in C^{\infty}_{0}\left(\Omega;\Lambda^{n-k-1}\right).$  The Hodge codifferential of $\omega\in L^{1}_{\rm loc}\left(\Omega;\Lambda^{k}\right)$ is
a $(k-1)$-form, denoted $\delta\omega\in L^{1}_{\rm loc}\left(\Omega;\Lambda^{k-1}\right)$
defined as
$$
\delta\omega:=(-1)^{nk+1}*d*\omega.
$$ See \cite{CsatoDacKneuss} for the properties and the integration by parts formula regarding these operators.
\item Let $ 1 \leq p \leq \infty$ and let $\nu$ be the outward unit normal to $\partial\Omega,$ identified with the $1$-form 
$\displaystyle \nu = \sum_{i=1}^{n} \nu_{i} dx^{i} .$ For any differential $k$-form $\omega$ on $\Omega,$ the $(k+1)$-form $\nu\wedge\omega$ and the $(k-1)$-form $\nu\lrcorner\omega$
 on $\partial\Omega$ are called the tangential and normal part of $\omega,$ respectively and are interpreted in the sense of traces when $\omega$ is not continuous. The spaces $W_{T}^{1,p}\left(  \Omega;\Lambda^{k}\right)  $ and
$W_{N}^{1,p}\left(  \Omega;\Lambda^{k}\right)  $ are defined as%
\[
W_{T}^{1,p}\left(  \Omega;\Lambda^{k}\right)  =\left\{  \omega\in
W^{1,p}\left(  \Omega;\Lambda^{k}\right)  :\nu\wedge\omega=0\text{ on
}\partial\Omega\right\}
\]%
\[
W_{N}^{1,p}\left(  \Omega;\Lambda^{k}\right)  =\left\{  \omega\in
W^{1,p}\left(  \Omega;\Lambda^{k}\right)  :\nu\,\lrcorner\,\omega=0\text{ on
}\partial\Omega\right\}.
\]
Also, we define, 
   $$W_{\delta, T}^{1,p}(\Omega; \Lambda^{k}) = \left\lbrace \omega \in W_{T}^{1,p}(\Omega; \Lambda^{k}) : \delta\omega = 0 \text{ in }
   \Omega \right\rbrace $$
   and $W_{\delta,N}^{1,p}(\Omega; \Lambda^{k}) $ is defined similarly. The space $\mathcal{H}_{T}\left(  \Omega;\Lambda^{k}\right)$ is defined as
\[
\mathcal{H}_{T}\left(  \Omega;\Lambda^{k}\right)  =\left\{  \omega\in
W_{T}^{1,2}\left(  \Omega;\Lambda^{k}\right)  :d\omega=0\text{ and }%
\delta\omega=0\text{ in }\Omega\right\}
\]%
The spaces $\mathcal{H}_{N}\left(  \Omega;\Lambda^{k}\right)  $ is defined similarly.

\item Let $B_{R}^{+}$ denote the half-ball centered around $0$ in the half space, i.e 
$$B_{R}^{+} = \lbrace  x \in \mathbb{R}^{n}: \lvert x \rvert < R, x_{n} > 0\rbrace .$$
Let $\Gamma_{R}$  and $C_{R}$ denote the flat part and the curved part, respectively, of the boundary of the half ball $B_{R}^{+}.$ 
We define the following subspace of $W^{1,2}(B_{R}^{+} ; \Lambda^{k})$,  $$W_{T, flat}^{1,2}(B_{R}^{+} ; \Lambda^{k})= \left\lbrace \psi \in W^{1,2}:
e_n \wedge \psi = 0 \text{ on }  \Gamma_{R},\text{ } \psi = 0 \text{ on } C_{R} \right\rbrace.$$
The space $W_{N, flat}^{1,2}(B_{R_{0}}^{+} ; \Lambda^{k})$ is defined similarly. We use the notation $\left( \cdot \right)_{s}$ to denote the average over half-balls, i.e 
$$ \left( \theta^{I} \right)_{s} = \frac{1}{\operatorname*{meas}(B_{s}^{+})} \int_{B^{+}_{s}} \theta^{I}.$$
\end{itemize}
We shall frequently use two ellipticity conditions for matrix fields. 
\begin{definition}\label{legendre-hadamard condition}
 A map $A:\Omega \rightarrow L(\Lambda^{k},\Lambda^{k})$ is said to satisfy the \textbf{Legendre-Hadamard condition} if $A$ satisfies,
 for all $x \in \Omega$,
 $$ \langle A (x) ( a\wedge b ) \  ;\  a\wedge b \rangle  \geq \gamma \left\vert  a \wedge b \right\vert^{2}, 
 \qquad  \text{ for every } a \in \Lambda^{1}, b \in \Lambda^{k-1}$$ for some constant $\gamma > 0.$ 
 $A$ is said to satisfy the \textbf{Legendre condition} if for some constant $\gamma > 0$ we have , for all  $x \in \Omega,$
 $$\langle A (x)  \xi \  ;\  \xi \rangle  \geq \gamma \left\vert  \xi \right\vert^{2}, 
 \qquad  \text{ for every } \xi \in \Lambda^{k}.$$ 
\end{definition}
Finally, to end this section, we shall frequently employ $c$ to denote a generic positive constant in the estimates, whose numerical value might change from one line to another. 

\section{Boundary estimates}\label{boundary estimates}

Our starting point is the classical Gaffney inequality, which can be proved from a simple integration by parts formula (see \cite{CsatoDacGaffney}). Using this inequality, we show 
that the usual `Campanato method' for regularity estimates (see for example, Giaquinta-Martinazzi\cite{giaquinta-martinazzi-regularity}) can be adapted to this setting to yield 
up to the boundary regularity estimates for \eqref{introhodgeelliptic} and \eqref{introhodgeellipticnormal} in the special case 
when $B \equiv \mathbf{I},$ the identity matrix. As a consequence, we prove a few general versions of the Gaffney-Friedrichs inequality which in turn yield the estimates for the 
general cases. The derivation of these inequalities is the main reason for proceeding in two stages rather than directly attempting to prove regularity in the general cases. 

Our strategy is to flatten the boundary and freeze the coefficients (the so-called `Korn's trick') and derive estimates in the Campanato spaces for the resulting constant 
coefficient operators on half-balls. 

\subsection{Preliminary lemmas}
We start with a few easy results that we shall use. 
\begin{lemma}[ellipticity lemma]\label{ellipticitylemma}
 Let $B:\Lambda^{k}\rightarrow \Lambda^{k}$ satisfy   
 $$ \langle B \xi ; \xi \rangle \geq \gamma_{B} \lvert \xi \rvert^{2} \qquad \text{ for all } \xi \in \Lambda^{k},$$ for some constant $ \gamma_{B} > 0.$ Then for every 
 $\gamma_{A} >0,$ there exist constants 
 $c_{1},c_{2}  > 0,$ depending only on $\gamma_{A}$ and $B,$ such that 
 \begin{align}
  \gamma_{A}\lvert e_{n} \wedge \xi \rvert^{2} + \left\lvert e_{n} \lrcorner B\xi \right\rvert^{2} 
&\geq c_{1} \left\lvert \xi \right\rvert^{2} \qquad \text{ for all } \xi \in \Lambda^{k}, \label{ellipticitytan}\\
 \gamma_{A}\lvert e_{n} \wedge B^{-1}\xi \rvert^{2} + \left\lvert e_{n} \lrcorner \xi \right\rvert^{2} 
&\geq c_{2} \left\lvert \xi \right\rvert^{2} \qquad \text{ for all } \xi \in \Lambda^{k}. \label{ellipticitynormal}
 \end{align}

\end{lemma}
\begin{proof} We prove by contradiction. If \eqref{ellipticitytan} is false then there exists a sequence $\lbrace \xi_{s} \rbrace$ such that for every $s \geq 1, $ we have $\lvert \xi_{s} \rvert =1 $ and  
 $\gamma_{A}\lvert e_{n} \wedge \xi_{s} \rvert^{2} + \left\lvert e_{n} \lrcorner B\xi_{s} \right\rvert^{2} 
< \frac{1}{s}. $ This implies, passing to a subsequence if necessary, $\xi_{s} \rightarrow \xi$ for some $\xi \in \Lambda^{k}$ such that 
$\lvert \xi \rvert =1 ,$ $e_{n} \wedge \xi = 0 $ and $e_{n} \lrcorner B\xi = 0. $ Since  
$ \xi = e_{n}\lrcorner \left( e_{n} \wedge \xi \right) + e_{n}\wedge \left(e_{n} \lrcorner \xi \right), $ this implies,
$$ 0 = \langle e_{n} \lrcorner B\xi ; e_{n} \lrcorner \xi \rangle = \langle   B\xi ; e_{n} \wedge \left( e_{n} \lrcorner \xi \right) \rangle 
 = \langle B \xi ; \xi \rangle \geq \gamma_{B} \lvert \xi \rvert^{2} =  \gamma_{B} > 0.$$ Similarly, if \eqref{ellipticitynormal} is false 
 then there exist $\xi \in \Lambda^{k}$ such that 
$\lvert \xi \rvert =1 ,$ $e_{n} \wedge B^{-1}\xi = 0 $ and $e_{n} \lrcorner \xi = 0. $ Using the same identity as above for $\xi$ implies, 
$$ 0 = \langle e_{n} \wedge B^{-1}\xi ; e_{n} \wedge \xi \rangle = \langle   B^{-1}\xi ; e_{n} \lrcorner \left( e_{n} \wedge \xi \right) \rangle 
 = \langle B^{-1} \xi ; \xi \rangle \geq  c \lvert \xi \rvert^{2} =  c  > 0,$$ since $B^{-1}$ satisfies a Legendre condition as well.  \end{proof}
 
The next result is a variant of classical G\aa{}rding inequality, which can be proved in the standard way. A detailed proof can be found in \cite{silthesis}.
\begin{lemma}[G\aa{}rding inequality]\label{garding inequality}
 Let $A \in C\left( \overline{\Omega}; L(\Lambda^{k+1},\Lambda^{k+1})\right)$ satisfy the Legendre-Hadamard condition. 
Then there exist constants $\lambda_{0} > 0$ and $\lambda_{1} \geq 0$ such that  for all $ u \in W_{T}^{1,2}(\Omega, \Lambda^{k}),$ we have,
\begin{equation}\label{garding inequality in Wd2}
 \int_{\Omega}  \langle A (x)d u  , d u \rangle \geq \lambda_{0} \left\Vert du \right\Vert_{L^{2}}^{2} 
  - \lambda_{1} \left\Vert u \right\Vert_{L^{2}}^{2} . 
\end{equation}
Moreover, if $A$ has constant coefficients or $A \in L^{\infty}\left(\Omega; L(\Lambda^{k+1},\Lambda^{k+1})\right)$ satisfies the Legendre condition, then the inequality is true with 
$\lambda_{1}=0.$
\end{lemma}
\begin{remark}\label{gardingremark}
 If $A$ satisfies only Legendre-Hadamard condition, estimate \eqref{garding inequality in Wd2} does not hold if we replace $W_{T}^{1,2}$ by $W_{N}^{1,2},$ even when $A$ has constant 
 coefficients. However, it is easy to see that if $A$ satisfies the Legendre condition, then the estimate holds with $\lambda_{1}=0$ for $u \in W_{N}^{1,2}(\Omega, \Lambda^{k})$ as well. 
\end{remark}

We now need a Poincar\'{e} inequality which can be proved by a simple contradiction argument followed by scaling.
\begin{lemma}
 There exist a constants $c_{1} ,c_{2} > 0$ such that
 \begin{align*}
  \int_{B^{+}_{R}} \lvert  u \rvert^{2} \leq c_{1}R^{2} \int_{B^{+}_{R}} \lvert \nabla u \rvert^{2}, \quad \text{ for all } u \in W_{T}^{1,2}(B_{R}^{+} ; \Lambda^{k})\cup 
  W_{N}^{1,2}(B_{R}^{+} ; \Lambda^{k})
\end{align*}
and 
\begin{align*}
 \int_{B^{+}_{R}} \lvert  u - \left( u \right)_{R} \rvert^{2} \leq c_{2}R^{2} \int_{B^{+}_{R}} \lvert \nabla u \rvert^{2}, \quad 
 \text{ for all } u \in W^{1,2}(B_{R}^{+} ; \Lambda^{k}).
\end{align*}
\end{lemma}
\subsection{Estimates for constant coefficient operator}
We begin by deriving estimates in half-balls for constant coefficient operators.
\subsubsection{Boundary Caccioppoli inequality}
The most crucial ingredients for these estimates are the Caccioppoli type inequalities on half-balls. 
\begin{theorem}[Caccioppoli inequality]\label{caccioppoli inequality half balls constant coefficients}
 Let $A:\Lambda^{k+1}\rightarrow \Lambda^{k+1}$ satisfy the Legendre-Hadamard condition. Let  $u \in W^{1,2}(B_{R}^{+} ; \Lambda^{k})$ with $\nu \wedge u = 0$ on $\Gamma_{R}$ satisfy, 
 for all $\psi \in W_{T, flat}^{1,2}(B_{R}^{+} ; \Lambda^{k}),$
 \begin{align}\label{constant coefficient boundary equation}
 \int_{B_{R}^{+}} \langle A(du) ; d\psi \rangle + \int_{B_{R}^{+}} \langle \delta u ; \delta \psi  \rangle = 0.
   \end{align}
   Then there exists a constant $c>0$ such that for every $0 < \rho  < R,$ we have the following 
  boundary Caccioppoli inequality,
  \begin{align}\label{boundary caccioppoli}
   \int_{B^{+}_{\rho}} \lvert \nabla u \rvert^{2} \leq \frac{c}{(R -\rho)^{2}}\left\lbrace \sum_{\substack{I \in \mathcal{T}^{k}\\ n \notin I}}
   \int_{B^{+}_{R}\setminus B^{+}_{\rho}} \left\lvert  u^{I} \right\rvert^{2} 
  + \sum_{\substack{I \in \mathcal{T}^{k}\\ n \in I}}\int_{B^{+}_{R}\setminus B^{+}_{\rho}} \left\lvert  u^{I} - \xi^{I}  \right\rvert^{2}\right\rbrace ,
  \end{align}
  for any collection of constants $\xi^{I} \in \mathbb{R}$ for each $I \in \mathcal{T}^{k}$ with $n \in I.$
\end{theorem}
\begin{proof}
 Choose $\eta \in C_{c}^{\infty}(B_{R})$ such that 
$ 0 \leq \eta \leq 1,$ $\eta \equiv 1 $ in $B_{\rho} $ and $ \lvert D\eta \rvert \leq \frac{2}{R - \rho}$ and
 define $\widetilde{u}: B^{+}_{R} \rightarrow \Lambda^{k}$ as,
\begin{align*}
 \widetilde{u}^{I} = \left\lbrace \begin{aligned}
                               &u^{I}  &\text{ if } n \notin I,\\
                               & u^{I} - \xi^{I}  \quad &\text{ if } n \in I,
                              \end{aligned}\right. \qquad \text{ for every } I \in \mathcal{T}^{k}.
\end{align*}
Substituting $\eta^{2}\widetilde{u} \in W_{T, flat}^{1,2}(B_{R}^{+} ; \Lambda^{k})$ 
for the test function $\psi$ in \eqref{constant coefficient boundary equation}, standard calculations and Young's inequality with $\varepsilon > 0$ implies the estimate, 
\begin{align*}
\int_{B_{R}^{+}} \langle A(d (\eta\widetilde{u})) ;  d(\eta\widetilde{u}) \rangle  
+ \int_{B_{R}^{+}}  \langle \delta (\eta \widetilde{u}) ; &  \delta ( \eta\widetilde{u} )  \rangle 
\\&\leq c\varepsilon \int_{B_{R}^{+}} \lvert \eta \rvert^{2} \lvert \nabla \widetilde{u} \rvert^{2}
+ c \int_{B_{R}^{+}} \lvert D\eta \rvert^{2} \lvert \widetilde{u} \rvert^{2} .
\end{align*}
Since the half ball $B_{R}^{+}$ is contractible, combining this with the Gaffney inequality and the G\aa{}rding inequality
\begin{align*}
&\int_{B_{R}^{+}} \lvert \eta \rvert^{2} \lvert \nabla \widetilde{u} \rvert^{2} = \int_{B_{R}^{+}} \lvert  \nabla \left( \eta \widetilde{u} \right) \rvert^{2} 
- \int_{B_{R}^{+}} \langle 2\eta \nabla \widetilde{u}; D\eta \otimes \widetilde{u} \rangle \\
&\leq c \left\lbrace \int_{B_{R}^{+}} \langle A(d (\eta\widetilde{u})) ;  d(\eta\widetilde{u}) \rangle  
+ \int_{B_{R}^{+}}  \langle \delta (\eta \widetilde{u}) ;  \delta ( \eta\widetilde{u} )  \rangle  \right\rbrace 
- \int_{B_{R}^{+}} \langle 2\eta \nabla \widetilde{u}; D\eta \otimes \widetilde{u} \rangle \\
&\leq c\varepsilon \int_{B_{R}^{+}} \lvert \eta \rvert^{2} \lvert \nabla \widetilde{u} \rvert^{2}
+ c \int_{B_{R}^{+}} \lvert D\eta \rvert^{2} \lvert \widetilde{u} \rvert^{2} 
\end{align*}
Choosing $\varepsilon$ small and using the facts that $\eta \equiv 1$ on $B_{\rho},$ $D\eta$ vanishes outside $B^{+}_{R} \setminus B^{+}_{\rho}$ 
and $\lvert D\eta \rvert \leq \frac{2}{R-\rho},$ we get 
\begin{align*}
 \int_{B^{+}_{\rho}} \lvert \nabla \widetilde{u} \rvert^{2} \leq \int_{B_{R}^{+}} \lvert \eta \rvert^{2} \lvert \nabla \widetilde{u} \rvert^{2} 
 \leq \frac{c}{ \left( R-\rho \right)^{2}} \int_{B^{+}_{R}\setminus B^{+}_{\rho}} \lvert \widetilde{u} \rvert^{2}  
 \leq \frac{c}{ \left( R-\rho \right)^{2}}  \int_{B^{+}_{R}\setminus B^{+}_{\rho}} \lvert \widetilde{u} \rvert^{2} .
\end{align*}
This finishes the proof. 
\end{proof}\smallskip 

\begin{remark}\label{caccioppoli normal}
If $A$  satisfies the Legendre condition, then similar arguments show that for any $u \in W^{1,2}(B_{R}^{+} ; \Lambda^{k})$ with $\nu \lrcorner u = 0$ on $\Gamma_{R}$ satisfying 
\eqref{constant coefficient boundary equation} for all $\psi \in W_{N, flat}^{1,2}(B_{R}^{+} ; \Lambda^{k})$ and for every $0 < \rho  < R,$ we have 
\begin{align*}
   \int_{B^{+}_{\rho}} \lvert \nabla u \rvert^{2} \leq \frac{c}{(R -\rho)^{2}}\left\lbrace \sum_{\substack{I \in \mathcal{T}^{k}\\ n \notin I}}
   \int_{B^{+}_{R}\setminus B^{+}_{\rho}} \left\lvert  u^{I} - \xi^{I}  \right\rvert^{2} 
  + \sum_{\substack{I \in \mathcal{T}^{k}\\ n \in I}}\int_{B^{+}_{R}\setminus B^{+}_{\rho}} \left\lvert  u^{I} \right\rvert^{2} \right\rbrace ,
  \end{align*}
  for any collection of constants $\xi^{I} \in \mathbb{R}$ for each $I \in \mathcal{T}^{k}$ with $n \notin I.$
  In view of remark \ref{gardingremark}, the stronger ellipticity assumption on $A$ is necessary.
\end{remark}

\subsubsection{$L^{2}$ estimates}
Caccioppoli inequalities lead to $L^2$ estimates.
\begin{theorem}[$L^{2}$ estimates for constant coefficients]\label{L2estimatesconstantcoeff}
  Let $A:\Lambda^{k+1}\rightarrow \Lambda^{k+1}$ satisfy the Legendre-Hadamard condition. Let  $u \in W^{1,2}(B_{R}^{+} ; \Lambda^{k})$ with $\nu \wedge u = 0$ on $\Gamma_{R}$  satisfy, 
 for all $\psi \in W_{T, flat}^{1,2}(B_{R}^{+} ; \Lambda^{k}),$
 \begin{align}\label{l2estimateconstantequation}
 \int_{B_{R}^{+}} \langle A(du) ; d\psi \rangle + \int_{B_{R}^{+}} \langle \delta u ; \delta \psi  \rangle = 0.
   \end{align}
  Then for every $m \geq 2,$ $u \in W^{m,2}(B_{R/2^{m}}^{+} ; \Lambda^{k})$ and there exist constants $c_{m}, \widetilde{c}_{m} > 0$ such that  we have the estimates
  \begin{equation}\label{l2estimate}
   \int_{B^{+}_{R/2^{m}}} \lvert D^{m} u \rvert^{2} \leq \frac{\widetilde{c}_{m}}{(R)^{2^{m-1}}}   \int_{B^{+}_{R}} \left\lvert \nabla u \right\rvert^{2} \leq \frac{c_{m}}{(R)^{2^{m}}}   \int_{B^{+}_{R}} \left\lvert  u \right\rvert^{2}. 
  \end{equation}
  \end{theorem}
\begin{remark}
 Same estimates hold true for $u \in W^{1,2}(B_{R}^{+} ; \Lambda^{k})$ with $\nu \lrcorner u = 0$ on $\Gamma_{R}$, satisfying \eqref{l2estimateconstantequation} for every 
 $\psi \in W_{N, flat}^{1,2}(B_{R}^{+} ; \Lambda^{k})$ if $A$ satisfies the Legendre condition, in view of remark \ref{caccioppoli normal}.
\end{remark}

\begin{proof} We only show \eqref{l2estimate} for $m=2.$ Using the difference quotient $\tau_{h,s}u(x) = \frac{1}{h}\left\lbrace 
u (x + he_{s}) - u(x)\right\rbrace ,$ for every $s=1,\ldots, n-1,$ and the Caccioppoli inequality for $\tau_{h,s}u$ and $u$, we obtain, 
\begin{align*}
 \int_{B^{+}_{\frac{R}{2^{2}}}} \left\lvert \nabla \left( \tau_{h,s}u \right) \right\rvert^{2} 
 \leq \frac{c}{(R)^2}   \int_{B^{+}_{R/2}} \left\lvert \tau_{h,s}u \right\rvert^{2} \leq \frac{c}{(R)^2}   \int_{B^{+}_{R/2}} \lvert \nabla u \rvert^{2}
 \leq \frac{c}{(R)^4} \int_{B^{+}_{R}} \left\lvert  u \right\rvert^{2}.
\end{align*}
This implies, 
 \begin{equation*}
   \int_{B^{+}_{R/2^{2}}} \lvert D_{ij} u \rvert^{2}  \leq \frac{c}{(R)^2}   \int_{B^{+}_{R/2}} \lvert \nabla u \rvert^{2} \leq \frac{c}{(R)^{2^{2}}}   \int_{B^{+}_{R}} \left\lvert  u \right\rvert^{2},
  \end{equation*}
  for all $i,j =1, \ldots, n$ such that $(i,j) \neq (n,n).$
  To estimate the $D_{nn}$ derivative, first we rewrite the system \eqref{l2estimateconstantequation} as a system in terms of gradients.
  For every $\psi \in  W_{T, flat}^{1,2}(B_{R}^{+} ; \Lambda^{k}),$ we have, 
  \begin{align}\label{gradient system}
   \int_{B_{R}^{+}} \langle A(du) ; d\psi \rangle + \int_{B_{R}^{+}} \langle \delta u ; \delta \psi  \rangle 
   = \int_{B_{R}^{+}} \langle \widetilde{A}\nabla u ; \nabla \psi \rangle, 
  \end{align}
where $\widetilde{A}:\mathbb{R}^{\tbinom{n}{k}\times n} \rightarrow \mathbb{R}^{\tbinom{n}{k}\times n}$ is the linear map, defined by the pointwise algebraic identities  
$$\langle \widetilde{A} a_{1} \otimes b_{1} ;  a_{2} \otimes b_{2} \rangle = \langle A\left( a_{1} \wedge b_{1} \right) ; a_{2} \wedge b_{2} \rangle 
+ \langle a_{1} \lrcorner b_{1} ; a_{2} \lrcorner b_{2}  \rangle,  $$    for every $ a_{1}, a_{2} \in \Lambda^{1}$, viewed also as vectors in $\mathbb{R}^{n}$  and $b_{1}, b_{2} \in \Lambda^{k},$
viewed also as vectors in $\mathbb{R}^{\tbinom{n}{k}}$. Given $\widetilde{A},$ we also define the maps $\widetilde{A}^{pq}:\mathbb{R}^{\tbinom{n}{k}} \rightarrow \mathbb{R}^{\tbinom{n}{k}} $ for every $p,q= 1,\ldots, n,$
 by the identities, 
 \begin{equation}\label{Apqdefinition} \langle \widetilde{A}^{pq}\xi ;\eta \rangle =\sum_{\alpha,\beta \in \mathcal{T}^{k}} \widetilde{A}^{pq}_{\alpha\beta} 
 \xi^{\alpha}\eta^{\beta} 
 = \langle \widetilde{A} (e_{p} \otimes \xi) ; e_{q} \otimes \eta \rangle, \qquad 
 \text{ for every } \xi,\eta \in \mathbb{R}^{\tbinom{n}{k}}.\end{equation}
Now, for every $\xi \in \mathbb{R}^{\tbinom{n}{k}},$ by lemma \ref{ellipticitylemma}, we get,
  \begin{align*}
  \langle \widetilde{A}^{nn} \xi ; \xi \rangle &= \langle \widetilde{A} (e_{n} \otimes \xi) ; e_{n} \otimes \xi \rangle = \langle A\left( e_{n} \wedge \xi \right) ; e_{n} \wedge \xi \rangle 
+ \langle e_{n} \lrcorner \xi ; e_{n} \lrcorner \xi \rangle \\&\geq \gamma_{A}\lvert e_{n} \wedge \xi \rvert^{2} + \left\lvert e_{n} \lrcorner \xi \right\rvert^{2} 
\geq c_{1} \left\lvert \xi \right\rvert^{2}, 
 \end{align*}
 proving that $\widetilde{A}^{nn}$ is invertible.  
 Choosing $\psi \in C_{c}^{\infty}(B_{R/2^{2}}^{+} ; \Lambda^{k})$ in \eqref{l2estimateconstantequation}, using \eqref{gradient system} and integrating by parts,   
 $u \in W^{1,2}(B_{R}^{+} ; \Lambda^{k})$  satisfy,  for all $\psi \in C_{c}^{\infty}(B_{R/2^{2}}^{+} ; \Lambda^{k}),$
 \begin{align*}
  \int_{B_{R/2^{2}}^{+}}\sum_{\alpha,\beta \in \mathcal{T}^{k}}   \widetilde{A}^{nn}_{\alpha\beta}  \frac{ \partial u^{\alpha}}{\partial x_n}\frac{\partial \psi^{\beta} }{\partial x_n}
      &= - \sum_{ \substack{ p,q = 1, \ldots, n \\ (p,q)\neq (n,n) \\\alpha,\beta \in \mathcal{T}^{k}}} \int_{B_{R/2^{2}}^{+}} \widetilde{A}^{pq}_{\alpha\beta} 
   \frac{ \partial u^{\alpha}}{\partial x_p}\frac{\partial \psi^{\beta} }{\partial x_q} \\ &=\sum_{ \substack{ p,q = 1, \ldots, n \\ (p,q)\neq (n,n) 
   \\\alpha,\beta \in \mathcal{T}^{k}}} 
  \int_{B_{R/2^{2}}^{+}} \frac{\partial}{\partial x_{q}}\left( \widetilde{A}^{pq}_{\alpha\beta} 
  \frac{ \partial u^{\alpha}}{\partial x_p}\right) \psi^{\beta}.
  \end{align*}
Since $\frac{\partial}{\partial x_{q}}\left( \widetilde{A}^{pq}_{\alpha\beta} 
  \frac{ \partial u^{\alpha}}{\partial x_p}\right) \in L^{2} $ for all choice of $p,q = 1, \ldots, n , (p,q)\neq (n,n) $ and for all $\alpha,\beta \in \mathcal{T}^{k},$ 
 and $\widetilde{A}^{nn}$ is invertible, this implies the estimate for $D_{nn} u.$  \end{proof}
\subsubsection{Decay estimates for the constant coefficient operator}
\begin{theorem}\label{decay estimates}
Let $A:\Lambda^{k+1}\rightarrow \Lambda^{k+1}$ satisfy the Legendre-Hadamard condition. Let  $u \in W_{T, flat}^{1,2}(B_{R}^{+} ; \Lambda^{k})$  satisfy, 
 for all $\psi \in W_{T, flat}^{1,2}(B_{R}^{+} ; \Lambda^{k}),$
 \begin{align}\label{decayestimateconstantequation}
 \int_{B_{R}^{+}} \langle A(du) ; d\psi \rangle + \int_{B_{R}^{+}} \langle \delta u ; \delta \psi  \rangle = 0.
   \end{align}
Then, for every $0 < \rho < R,$ the following decay estimates hold true.
 \begin{align}
  \int_{B^{+}_{\rho}} \left\lvert  u \right\rvert^{2} &\leq c \left( \frac{\rho}{R}\right)^{n} \int_{B^{+}_{R}} \left\lvert  u \right\rvert^{2}, \label{decay of u}\\
  \int_{B^{+}_{\rho}} \left\lvert \nabla u \right\rvert^{2} &\leq c \left( \frac{\rho}{R}\right)^{n} \int_{B^{+}_{R}} \left\lvert   \nabla u \right\rvert^{2} , 
  \label{decay of Du}\\
 \int_{B^{+}_{\rho}} \left\lvert \nabla u -  \left( \nabla u \right)_{\rho} \right\rvert^{2} &\leq c \left( \frac{\rho}{R}\right)^{n+2} 
 \int_{B^{+}_{R}} \left\lvert \nabla u -  \left( \nabla u \right)_{R} \right\rvert^{2}. \label{decay of Du - Du average}
 \end{align}
\end{theorem}
\begin{remark}
 As before, analogous estimates hold for $u \in W_{N, flat}^{1,2}$  satisfying \eqref{decayestimateconstantequation} for 
 all $\psi \in W_{N, flat}^{1,2}(B_{R}^{+} ; \Lambda^{k}),$ when $A$ satisfies the Legendre condition.
\end{remark}

\begin{proof}
Using the rescaled function $\widetilde{u}(x) = u (Rx),$ we can replace $\rho$ by $\rho / R$ and $R$ by $1.$ Since the inequalities are nontrivial 
only for $\frac{\rho}{R}$ small, we assume $\frac{\rho}{R} < \frac{1}{2^{m}}$ for some integer $m.$ Now by Sobolev embedding and $L^{2}$ regularity results give,  
\begin{align*}
 \int_{B^{+}_{\rho / R}} \left\lvert  u \right\rvert^{2} \leq c\left( \frac{\rho}{R}\right)^{n} \sup_{B^{+}_{\rho / R}} \left\lvert  u \right\rvert^{2} 
 \leq c\left( \frac{\rho}{R}\right)^{n} \lVert u \rVert^{2}_{W^{m_{0},2}(B^{+}_{\frac{1}{2^{m_{0}}}})} \leq c \left( \frac{\rho}{R}\right)^{n} \int_{B^{+}_{1}} \left\lvert  u \right\rvert^{2}.
\end{align*}
Same argument for $\nabla u$ and Poincar\'{e} inequality give \eqref{decay of Du}.
Now we define $\widetilde{u} \in W^{1,2}(B_{R}^{+} ; \Lambda^{k})$,
\begin{align*}
 \widetilde{u}^{I} = \left\lbrace \begin{aligned}
                               &u^{I}  &\text{ if } n \notin I,\\
                               & u^{I} - \sum_{i=1}^{n} x_{i}\left( \frac{\partial u^{I}}{\partial x_{i}} \right)_{B^{+}_{1}}   \quad &\text{ if } n \in I,
                              \end{aligned}\right. \qquad \text{ for every } I \in \mathcal{T}^{k}.
\end{align*}
Note that $D^{2}\widetilde{u} = D^{2}u, $ $\widetilde{u}$ satisfies the same PDE as $u$ and $\nu \wedge \widetilde{u} = 0$ on 
$\Gamma_{R}.$ Thus the $L^{2}$ regularity, Poincar\'{e} inequality and Sobolev embedding gives,
\begin{align*}
\int_{B_{\rho / R}} \left\lvert \nabla u -  \left( \nabla u \right)_{\rho / R} \right\rvert^{2} &\leq c \left( \frac{\rho}{R}\right)^{2} \int_{B^{+}_{\rho / R}} \left\lvert  D^{2} u \right\rvert^{2} 
  \leq c \left( \frac{\rho}{R}\right)^{n+2}  \sup_{B^{+}_{1 /2^{m_{1}}}} \left\lvert  D^{2} u \right\rvert^{2} \\
  &= c \left( \frac{\rho}{R}\right)^{n+2}  \sup_{B^{+}_{1 /2^{m_{1}}}} \left\lvert  D^{2} \widetilde{u} \right\rvert^{2} \leq c \left( \frac{\rho}{R}\right)^{n+2}\left\lVert D\widetilde{u}\right\rVert^{2}_{L^{2}(B^{+}_{1})}.
\end{align*}
But we have
$$\left\lVert D\widetilde{u}\right\rVert^{2}_{L^{2}(B^{+}_{1})} = \int_{B^{+}_{1}} \left\lvert \nabla u -  \left( \nabla u \right)_{1} \right\rvert^{2},$$
since $\left( \nabla u_{I} \right)_{1} = 0$ for all $I \in \mathcal{T}^{k}$ with $n \notin I.$  \end{proof}
\subsubsection{$\mathcal{L}^{2, \mu}$ estimates}
Now we proceed to estimates in the Campanato spaces. 
\begin{theorem}[$\mathcal{L}^{2,\mu}$ estimates]\label{l2mu estimates}
 Let $0 \leq \mu \leq n+ 2\gamma $ be a real number. Let the matrix $A:\Lambda^{k+1}\rightarrow \Lambda^{k+1}$ satisfy the Legendre-Hadamard condition and $u \in W_{T, flat}^{1,2}(B_{R}^{+} ; \Lambda^{k})$  satisfy, 
 for all $\psi \in W_{T, flat}^{1,2}(B_{R}^{+} ; \Lambda^{k}),$
 \begin{align}\label{l2muequation}
 \int_{B_{R}^{+}} \langle A(du) ; d\psi \rangle + \int_{B_{R}^{+}} \langle \delta u ; \delta \psi  \rangle &+  \int_{B_{R}^{+}} \langle \mathcal{P} ;   \psi \rangle 
  -\int_{B_{R}^{+}} \langle \mathcal{Q} ; \nabla\psi \rangle = 0.
   \end{align}
If we have $\mathcal{P} \in \mathrm{L}^{2, \mu - 2}(B^{+}_{R}; \Lambda^{k})$ and $\mathcal{Q} \in \mathcal{L}^{2, \mu}(B^{+}_{R}; \mathbb{R}^{n \times \tbinom{n}{k}}) ,$ then   
 $Du \in \mathcal{L}^{2, \mu}\left( B^{+}_{R / 2}; \mathbb{R}^{n \times \tbinom{n}{k}}\right) ,$ 
with the estimate
\begin{equation*}
 \left[ Du \right]_{\mathcal{L}^{2, \mu}\left( B^{+}_{R / 2}\right)} \leq  c \left\lbrace \lVert Du \rVert_{L^{2}(B_{R}^{+})} 
 + \left[ \mathcal{P} \right]_{\mathrm{L}^{2, \mu - 2}(B^{+}_{R})}  
 + \left[ \mathcal{Q} \right]_{\mathcal{L}^{2, \mu}(B^{+}_{R})}  \right\rbrace .
\end{equation*}
Moreover, if $\mathcal{P} \in \mathcal{L}^{2, \mu}(B^{+}_{R}; \Lambda^{k})$ and $\mathcal{Q} \in \mathcal{L}^{2, \mu +2 }(B^{+}_{R}; \mathbb{R}^{n \times \tbinom{n}{k}}),$ 
then we also have $D^{2}u \in \mathcal{L}^{2, \mu}\left( B^{+}_{R / 2}; \mathbb{R}^{n^{2} \times \tbinom{n}{k}}\right) ,$ with the estimate 
\begin{equation*}
 \left[ D^{2} u \right]_{\mathcal{L}^{2, \mu}\left( B^{+}_{R / 2}\right)} \leq  c \left\lbrace \lVert Du \rVert_{L^{2}} 
 + \left[ \mathcal{P} \right]_{\mathcal{L}^{2, \mu}(B^{+}_{R})}  
 + \left[ \mathcal{Q} \right]_{\mathcal{L}^{2, \mu +2 }(B^{+}_{R})}  \right\rbrace .
\end{equation*}
\end{theorem}
\begin{remark}
The same estimates hold as well if $A$ satisfies the Legendre condition and $u \in W_{N, flat}^{1,2}(B_{R}^{+} ; \Lambda^{k})$  satisfy \eqref{l2muequation} for 
all $\psi \in W_{N, flat}^{1,2}(B_{R}^{+} ; \Lambda^{k}).$
\end{remark}
\begin{proof} We just show the estimate of $$\displaystyle \sup_{0 < \rho < \delta R/2}\frac{1}{\rho^{\mu}}\int_{B^{+}_{\rho}} \lvert Du - \left( Du \right)_{\rho} \rvert^{2},$$
 with $0 < \delta < \frac{1}{2}$ fixed. We write $u = v+ w$ where $v \in W^{1,2}(B_{R}^{+} ; \Lambda^{k})$ solves
\begin{align*}
 \left\lbrace \begin{aligned}
               \delta (A dv) + d\delta v &= 0 \quad &&\text{ in } B^{+}_{R}, \\
   v &=  u \qquad &&\text{ on } \partial B^{+}_{R}.
              \end{aligned} \right. \end{align*}
We have, by standard arguments using theorem \ref{decay estimates} for $v,$  
\begin{align}\label{phi1testimate}
 \int_{B^{+}_{\rho}} \lvert Du - \left( Du \right)_{\rho} \rvert^{2} 
 &\leq c\left( \frac{\rho}{R}\right)^{n+2} \int_{B^{+}_{R}} \lvert Du - \left( Du \right)_{R} \rvert^{2} +  c_{1} \int_{B^{+}_{R}} \lvert \nabla w \rvert^{2}. 
\end{align}
Note that $w$ satisfies
\begin{align*}
 \int_{B_{R}^{+}} \langle A(dw) ; d w \rangle + \int_{B_{R}^{+}} \langle \delta w ; \delta w  \rangle + & \int_{B_{R}^{+}} \langle \mathcal{P} ;   w \rangle 
  -\int_{B_{R}^{+}} \langle \mathcal{Q} ; \nabla w \rangle     = 0.
   \end{align*}
Using H\"{o}lder inequality, Poincar\'{e} inequality and Young's inequality with $\varepsilon,$
\begin{align*}
 \left\lvert \int_{B_{R}^{+}} \langle \mathcal{P} ; w \rangle  \right\rvert &\leq  \left( \int_{B_{R}^{+}}\lvert \mathcal{P} \rvert^{2}\right)^{\frac{1}{2}}
 \left( c R^{2} \int_{B_{R}^{+}}\lvert \nabla w \rvert^{2} \right)^{\frac{1}{2}} \\& \leq \varepsilon\int_{B_{R}^{+}}\lvert \nabla w \rvert^{2} 
 + c [ \mathcal{P} ]^{2}_{\mathrm{L}^{2, \mu - 2 }} R^{\mu}.
\end{align*}
Thus, since $ w = 0$ on all of $\partial B^{+}_{R},$ using Young's inequality  with $\varepsilon$ we deduce, 
\begin{align*}
\left\lvert \int_{B_{R}^{+}} \langle \mathcal{Q} ; \nabla w \rangle  \right\rvert  
= \left\lvert \int_{B_{R}^{+}} \left\langle \mathcal{Q} - \left( \mathcal{Q}\right)_{R} ; \nabla w \right\rangle  \right\rvert
\leq \varepsilon\int_{B^{+}_{R}} \lvert \nabla w \rvert^{2} + c \left[ \mathcal{Q} \right]^{2}_{\mathcal{L}^{2, \mu}(B^{+}_{R})} R^{\mu} .
\end{align*}
Now, by the Gaffney inequality and the G\aa{}rding inequality, we have, 
\begin{align*}
 \int_{B^{+}_{R}} \lvert \nabla w \rvert^{2} \leq c \left\lbrace \int_{B_{R}^{+}} \lvert d w \rvert^{2} + \int_{B_{R}^{+}} \lvert \delta w \rvert^{2} \right\rbrace 
 \leq c \left\lbrace \int_{B_{R}^{+}} \langle A(dw) ; dw \rangle + \int_{B_{R}^{+}} \lvert \delta w \rvert^{2} \right\rbrace.
\end{align*}
Combining this with the above estimates and choosing $\varepsilon$ to be small enough, we obtain an estimate for $\int_{B^{+}_{R}} \lvert \nabla w \rvert^{2}.$ 
Combining with \eqref{phi1testimate} gives
\begin{multline*}
 \int_{B^{+}_{\rho}} \lvert Du - \left( Du \right)_{\rho} \rvert^{2}\leq c\left( \frac{\rho}{R}\right)^{n+2}\int_{B^{+}_{R}} \lvert Du - \left( Du \right)_{R} \rvert^{2} \\ +  c  R^{\mu} \left(  \left[ \mathcal{P} \right]^{2}_{\mathrm{L}^{2, \mu - 2 }} 
  +\left[ \mathcal{Q} \right]^{2}_{\mathcal{L}^{2, \mu}}    \right).
\end{multline*}
By usual arguments utilizing the scaling lemma (see lemma 5.13 in \cite{giaquinta-martinazzi-regularity}),
$$\frac{1}{\rho^{\mu}}\int_{B^{+}_{\rho}} \lvert Du - \left( Du \right)_{\rho} \rvert^{2} \leq c\left( \left\lVert Du \right\rVert_{L^{2}(B_{+}^{R})}^{2} +     \left[ \mathcal{P} \right]^{2}_{\mathrm{L}^{2, \mu - 2 }} 
  +\left[ \mathcal{Q} \right]^{2}_{\mathcal{L}^{2, \mu}}    \right).$$ This yields the desired estimate. 
\end{proof}
\subsubsection{$L^{p}$ estimates}
\begin{theorem}[$L^{p}$ estimates]\label{lpestimatesconstantcoeff}
 Let $ 1 <  p < \infty$ and $1 < q < \infty$ be such that  $q^{*} =\frac{nq}{n-q} \geq p$ if $q < n.$ Let $A:\Lambda^{k+1}\rightarrow \Lambda^{k+1}$ satisfy 
 the Legendre-Hadamard condition. Let $u \in W_{T, flat}^{1,2}(B_{R}^{+} ; \Lambda^{k})$  satisfy, 
 for all $\psi \in W_{T, flat}^{1,2}(B_{R}^{+} ; \Lambda^{k})\cap W^{1,p^{'}}(B_{R}^{+} ; \Lambda^{k}) ,$
 \begin{align}\label{Lpestimateconstant}
 \int_{B_{R}^{+}} \langle A(du) ; d\psi \rangle + \int_{B_{R}^{+}} \langle \delta u ; \delta \psi  \rangle &+  \int_{B_{R}^{+}} \langle \mathcal{P} ;   \psi \rangle 
  -\int_{B_{R}^{+}} \langle \mathcal{Q} ; \nabla\psi \rangle     = 0.
   \end{align}
If $\mathcal{P} \in L^{q}(B^{+}_{R}; \Lambda^{k})$ and  $\mathcal{Q} \in L^{p}(B^{+}_{R}; \mathbb{R}^{n \times \tbinom{n}{k}}), $ 
then $Du \in L^{p}\left( B^{+}_{R}; \mathbb{R}^{n \times \tbinom{n}{k}}\right) $ 
and we have the estimate
\begin{equation*}
 \left\lVert Du \right\rVert_{L^{p}\left( B^{+}_{R}\right)} \leq  c \left\lbrace  
  \left\lVert \mathcal{P} \right\rVert_{L^{q}(B^{+}_{r})}  
 + \left\lVert \mathcal{Q} \right\rVert_{L^{p}(B^{+}_{r})}  \right\rbrace .
\end{equation*}
Moreover, if  $\mathcal{P} \in L^{p}(B^{+}_{R}; \Lambda^{k})$ and $\mathcal{Q} \in W^{1, p }(B^{+}_{r}; \mathbb{R}^{n \times \tbinom{n}{k}}),$ 
then we also have $D^{2}u \in L^{p}\left( B^{+}_{r / 2}; \mathbb{R}^{n^{2} \times \tbinom{n}{k}}\right) $ with the estimate 
\begin{equation*}
 \left\lVert D^{2} u \right\rVert_{L^{p}\left( B^{+}_{R}\right)} \leq  c \left\lbrace \left\lVert \mathcal{P} \right\rVert_{L^{p}(B^{+}_{r})}  
 + \left\lVert \mathcal{Q} \right\rVert_{W^{1, p }(B^{+}_{r})}  \right\rbrace .
\end{equation*}
\end{theorem}
\begin{remark}
The same estimates hold as well if $A$ satisfies the Legendre condition and $u \in W_{N, flat}^{1,2}(B_{R}^{+} ; \Lambda^{k})$  satisfy \eqref{Lpestimateconstant} for 
all $\psi \in W_{N, flat}^{1,2}(B_{R}^{+} ; \Lambda^{k}).$
\end{remark}
\begin{proof}
The only difference from standard methods lies in the argument to show that $\mathcal{P}$ can be absorbed in the $\mathcal{Q}$ term and we can assume $\mathcal{P} = 0.$

\noindent For every $I \in \mathcal{T}^{k}$ such that $n \in I,$ we solve the following Neumann boundary value problem for 
the scalar Laplacian,
\begin{align*}
 \left\lbrace \begin{aligned}
               -\Delta \theta_{I} &= \mathcal{P}_{I} \qquad \text{ in } B^{+}_{R}, \\
  \frac{\partial \theta_{I}}{\partial \nu} &= g_{I} \qquad \text{ in } \partial B^{+}_{R}, 
              \end{aligned}\right.\end{align*}
where the functions $g_{I}$s are such that $$ \int_{\partial B^{+}_{R}} g = \int_{B^{+}_{R}} \mathcal{P}_{I} \qquad \text{ and } \qquad g_{I} = 0 \text{ on } \Gamma_{R}^{+}.$$
For every $I \in \mathcal{T}^{k}$ such that $n \notin I,$ we solve the following Dirichlet boundary value problem for 
the scalar Laplacian,
\begin{align*}
 \left\lbrace \begin{aligned}
               -\Delta \theta_{I} &= \mathcal{P}_{I} \qquad \text{ in } B^{+}_{R}, \\
  \theta_{I} &= 0 \qquad \text{ in } \partial B^{+}_{R}. 
              \end{aligned}\right.\end{align*}
Since $1 < q < \infty,$ by regularity results for the scalar Laplacian, we deduce that $\theta_{I} \in W^{2,q}\left( B^{+}_{R} \right)$ along with the estimate 
$$ \left\lVert \theta_{I}\right\rVert_{W^{2,q }} \leq \left\lVert \mathcal{P}_{I}\right\rVert_{L^{q }}. $$

Now, using integrating by parts, we can rewrite the term $\int_{B_{R}^{+}} \left\langle \mathcal{P}; \psi\right\rangle $ as, 
\begin{align*}
 \int_{B_{R}^{+}} \left\langle \mathcal{P}; \psi\right\rangle &= \int_{B_{R}^{+}} \sum_{I \in \mathcal{T}^{k}} \mathcal{P}_{I}\psi_{I} =
 -\int_{B_{R}^{+}} \sum_{I \in \mathcal{T}^{k}} \Delta \theta_{I}\psi_{I} \\&= \int_{B_{R}^{+}} \sum_{I \in \mathcal{T}^{k}}  \left\langle \nabla \theta_{I} ; \nabla\psi_{I} \right\rangle 
 - \int_{\partial B_{R}^{+}} \sum_{I \in \mathcal{T}^{k}}  \frac{\partial \theta_{I}}{\partial \nu}\psi_{I} \\
 &=\int_{B_{R}^{+}} \sum_{I \in \mathcal{T}^{k}}  \left\langle \nabla \theta_{I} ; \nabla\psi_{I} \right\rangle 
 - \int_{\Gamma_{R}^{+}} \sum_{\substack{n \in I \\I \in \mathcal{T}^{k}}}  \frac{\partial \theta_{I}}{\partial \nu}\psi_{I} \\
 &=\int_{B_{R}^{+}} \sum_{I \in \mathcal{T}^{k}}  \left\langle \nabla \theta_{I} ; \nabla\psi_{I} \right\rangle.
\end{align*}
Thus this term can be absorbed in the $\mathcal{Q}$ term, since $\nabla \theta_{I} \in L^{p}$ as a consequence of Sobolev embedding and 
the fact that $\theta_{I} \in W^{2, q}$ with $q^{*} \geq p.$\smallskip

\noindent The rest is standard. By theorem \ref{l2mu estimates}, the linear map 
$\mathcal{Q} \mapsto T(\mathcal{Q}) = \nabla u,$ can be extended as a bounded linear 
map from $L^{2}( B^{+}_{R }; \mathbb{R}^{n \times \tbinom{n}{k}})$ to $L^{2}( B^{+}_{R }; \mathbb{R}^{n \times \tbinom{n}{k}})$ and from 
$L^{\infty}( B^{+}_{R }; \mathbb{R}^{n \times \tbinom{n}{k}})$ to $BMO( B^{+}_{R }; \mathbb{R}^{n \times \tbinom{n}{k}}) .$ Thus Stampacchia's interpolation theorem implies that the map extends as a bounded linear operator from 
$L^{p}( B^{+}_{R }; \mathbb{R}^{n \times \tbinom{n}{k}})$ to $L^{p}( B^{+}_{R }; \mathbb{R}^{n \times \tbinom{n}{k}}),$ for all $2 \leq p < \infty. $ 
The case $1 < p < 2$ follows by usual duality arguments.  \end{proof}
\subsection{Schauder and $L^{p}$ estimates}\label{schauderandlp}
We are now ready to derive the global estimates. We use lemma \ref{flattening and freezing coefficients} to flatten the boundary, whose proof has been relegated to the appendix. 
By usual patching arguments, everything boils down to proving the estimates for $u \in W_{T}^{1,2}(B_{R}^{+};\Lambda^{k})$ (respectively $W_{N}^{1,2}(B_{R}^{+};\Lambda^{k})$) 
satisfying 
\begin{align*}
 \int_{B_{R}^{+}} \langle \bar{A}(du) ; d\psi \rangle + \int_{B_{R}^{+}} \langle \delta u ; \delta \psi  \rangle +  \int_{B_{R}^{+}} \langle \mathcal{P} ;   \psi \rangle 
  -\int_{B_{R}^{+}} &\langle \mathcal{Q} ; \nabla\psi \rangle \notag \\&+ \int_{B_{R}^{+}} \langle \mathrm{S}\nabla u ; \nabla\psi \rangle= 0,
   \end{align*}
for all $\psi \in W_{T}^{1,2}(B_{R}^{+};\Lambda^{k}),$ (respectively $W_{N}^{1,2}(B_{R}^{+};\Lambda^{k})$),
where $\bar{A}$ satisfies the Legendre-Hadamard condition (respectively Legendre condition), 
$\mathcal{P} = \widetilde{f} + \mathrm{P}u + \mathrm{R} \nabla u$  and $\mathcal{Q} = \widetilde{F}  -  \mathrm{Q}u ,$ with $\widetilde{f}, \widetilde{F}, 
\mathrm{P}, \mathrm{Q}, \mathrm{R}, \mathrm{S}$ are as in lemma \ref{flattening and freezing coefficients}, for $R>0$ suitably small, 

\begin{theorem}[$C^{r+2,\gamma}$ regularity]\label{boundary Cralpha regularity linear}
 Let $1\leq k\leq n-1$, $r \geq 0$ be integers and  let $0 < \gamma < 1 $ be a real number. Let $\Omega
\subset\mathbb{R}^{n}$ be an open, bounded $C^{r+2,\gamma}$ set. Let $A  \in C^{r,\gamma}(\overline{\Omega} ; L(\Lambda^{k+1},\Lambda^{k+1}))$ 
satisfy the Legendre-Hadamard condition. 
Also  let $f \in C^{r, \gamma }(\overline{\Omega}, \Lambda^{k}) $, $ F \in C^{r, \gamma }(\overline{\Omega}, \Lambda^{k+1})$ and $\lambda \in \mathbb{R}.$ 
 Let  $\omega \in W_{T}^{1,2}(\Omega, \Lambda^{k})$ be a weak solution of the following,
\begin{equation}\label{schauder estimate equation}
 \int_{\Omega}  \langle A (x)d\omega , d\phi \rangle  + \int_{\Omega}\langle \delta \omega , \delta \phi \rangle 
  + \lambda  \int_{\Omega}\langle \omega , \phi \rangle 
  +\int_{\Omega} \langle  f, \phi \rangle  - \int_{\Omega} \langle  F, d\phi \rangle = 0,   
\end{equation}
for all $ \phi \in W_{T}^{1,2}(\Omega, \Lambda^{k})$.
Then $\omega \in C^{r+1,\gamma}(\overline{\Omega}, \Lambda^{k})$   and satisfies the estimate, 
\begin{equation*}
  \lVert  \omega \rVert_{C^{r+1,\gamma}(\overline{\Omega};\Lambda^{k})} \leq   c \left\lbrace \lVert \omega \rVert_{C^{0,\gamma}(\overline{\Omega};\Lambda^{k})}
   + \lVert f \rVert_{C^{r,\gamma}(\overline{\Omega};\Lambda^{k})}  + \lVert F \rVert_{C^{r,\gamma}(\overline{\Omega};\Lambda^{k+1})}\right\rbrace,
\end{equation*}
where the constant $c>0$ depends only on $A, \lambda, r$ and $ \Omega. $ Moreover, if $A  \in C^{r+1,\gamma}(\overline{\Omega} ; L(\Lambda^{k+1},\Lambda^{k+1}))$ and $ F \in C^{r+1, \gamma }(\overline{\Omega}, \Lambda^{k+1}),$ then 
$\omega \in C^{r+2,\gamma}(\overline{\Omega}, \Lambda^{k})$   and satisfies the estimate, 
\begin{equation*}
  \lVert  \omega \rVert_{C^{r+2,\gamma}(\overline{\Omega};\Lambda^{k})} \leq   c \left\lbrace \lVert \omega \rVert_{C^{0,\gamma}(\overline{\Omega};\Lambda^{k})}
   + \lVert f \rVert_{C^{r,\gamma}(\overline{\Omega};\Lambda^{k})}  + \lVert F \rVert_{C^{r+1,\gamma}(\overline{\Omega};\Lambda^{k+1})}\right\rbrace,
\end{equation*}
where the constant $c>0$ depends, once again, only on $A, \lambda, r$ and $ \Omega. $
\end{theorem}
\begin{remark}
 If $A$ satisfies the Legendre condition, same conclusions hold as well for any weak solution $\omega \in W_{N}^{1,2}(\Omega, \Lambda^{k})$ satisfying 
 \eqref{schauder estimate equation} for all $ \phi \in W_{N}^{1,2}(\Omega, \Lambda^{k})$.
\end{remark}
\begin{proof}
 Writing $u = v+ w ,$ where $v \in W^{1,2}(B_{R}^{+} ; \Lambda^{k})$ is the solution of 
the homogeneous equation for $\bar{A}$ coinciding with $u$ on the boundary and noting that  
\begin{align*}
  \left\lvert \int_{B_{R}^{+}} \langle \mathrm{S} \nabla u ; \nabla w \rangle  \right\rvert  
\leq \varepsilon\int_{B^{+}_{R}} \lvert \nabla w \rvert^{2} + c \left( \mathcal{C}_{A}^{2}(R) + \mathcal{C}_{\Phi}^{2}(R) \right)\int_{B^{+}_{R}} \lvert \nabla u \rvert^{2},
\end{align*}
the Schauder estimates follow from theorem \ref{l2mu estimates} by standard techniques (see \cite{giaquinta-martinazzi-regularity}) after absorbing the lower order 
terms by interpolation.  
\end{proof}

\begin{theorem}[$W^{r+2,p}$ regularity]\label{boundary Wrp regularity linear}
 Let $1\leq k\leq n-1$, $r \geq 0$ be integers and  let $1 < p < \infty $ and $1 < q < \infty$ be real numbers such that  $q^{*} =\frac{nq}{n-q} \geq p.$ Let $\Omega
\subset\mathbb{R}^{n}$ be an open, bounded $C^{r+2}$ set. Let $A  \in C^{r}(\overline{\Omega} ; L(\Lambda^{k+1},\Lambda^{k+1}))$ 
satisfy  the Legendre-Hadamard condition. 
Also  let $f \in W^{r, q }(\Omega, \Lambda^{k}) $, $ F \in W^{r, p }(\Omega, \Lambda^{k+1})$ and $\lambda \in \mathbb{R}.$ 
 Let  $\omega \in W_{T}^{1,2}(\Omega, \Lambda^{k})$ be a weak solution of the following,
\begin{equation}\label{Lp estimate equation}
 \int_{\Omega}  \langle A (x)d\omega , d\phi \rangle  + \int_{\Omega}\langle \delta \omega , \delta \phi \rangle 
  + \lambda  \int_{\Omega}\langle \omega , \phi \rangle 
  +\int_{\Omega} \langle  f, \phi \rangle  - \int_{\Omega} \langle  F, d\phi \rangle = 0,   
\end{equation}
for all $ \phi \in W_{T}^{1,2}(\Omega, \Lambda^{k}) \cap W^{1,p^{'}}(\Omega, \Lambda^{k})$.
Then $\omega \in W^{r+1,p}(\Omega, \Lambda^{k}),$   and satisfies the estimate, 
\begin{equation*}
  \lVert  \omega \rVert_{W^{r+1,p}(\Omega;\Lambda^{k})} \leq   c \left\lbrace \lVert \omega \rVert_{L^{p}(\Omega;\Lambda^{k})}
   + \lVert f \rVert_{W^{r,q}(\Omega;\Lambda^{k})}  + \lVert F \rVert_{W^{r,p}(\Omega;\Lambda^{k+1})}\right\rbrace,
\end{equation*}
where the constant $c>0$ depends only on $A, \lambda, r$ and $ \Omega. $ Moreover, if $A  \in C^{r+1}(\overline{\Omega} ; L(\Lambda^{k+1},\Lambda^{k+1})),$ 
$f \in W^{r, p }(\Omega, \Lambda^{k}) $ and 
$ F \in W^{r+1, p }(\Omega, \Lambda^{k+1}),$ then 
$\omega \in W^{r+2,p}(\Omega, \Lambda^{k}),$   and satisfies the estimate, 
\begin{equation*}
  \lVert  \omega \rVert_{W^{r+2,p}(\Omega;\Lambda^{k})} \leq   c \left\lbrace \lVert \omega \rVert_{L^{p}(\Omega;\Lambda^{k})}
   + \lVert f \rVert_{W^{r,p}(\Omega;\Lambda^{k})}  + \lVert F \rVert_{W^{r+1,p}(\Omega;\Lambda^{k+1})}\right\rbrace,
\end{equation*}
where the constant $c>0$ depends, once again, only on $A, \lambda, r$ and $ \Omega. $
\end{theorem}

\begin{remark}
 If $A$ satisfies the Legendre condition, same conclusions hold for any weak solution $\omega \in W_{N}^{1,2}(\Omega, \Lambda^{k})$ satisfying 
 \eqref{Lp estimate equation} for all $ \phi \in W_{N}^{1,2}(\Omega, \Lambda^{k}) \cap W^{1,p^{'}}(\Omega, \Lambda^{k})$.
\end{remark}

\begin{proof}
The result follows from Theorem \ref{lpestimatesconstantcoeff} by standard arguments.  
The crucial step is to show that if $u \in W^{1,m},$ for some $2 \leq m \leq p,$ then for every $V \in W^{1,s}_{T,flat},$ there is a unique 
solution $v \in W^{1,s}_{T,flat} $ satisfying, for every 
    $\psi \in W^{1,2}_{T,flat}(B_{R}^{+};\Lambda^{k})\cap W^{1,p^{'}}(B_{R}^{+};\Lambda^{k}),$ 
   \begin{align*}
 \int_{B_{R}^{+}} \langle \bar{A}(dv) ; d\psi \rangle + \int_{B_{R}^{+}} \langle \delta v ; \delta \psi  \rangle 
 &+  \int_{B_{R}^{+}} \langle \widetilde{f} + \mathrm{P}u + \mathrm{R} \nabla u ;   \psi \rangle 
   \\&-\int_{B_{R}^{+}} \langle \widetilde{F}  -  \mathrm{Q}u ; \nabla\psi \rangle + \int_{B_{R}^{+}} \langle \mathrm{S}\nabla V ; \nabla\psi \rangle= 0,
   \end{align*}
   where $s = \min \left( p , m^{*} \right).$ Then we show that the map 
   $T:W^{1,s}_{T,flat} \rightarrow W^{1,s}_{T,flat} $ that maps $V$ to $v$ has a fixed point if $R$ is chosen small enough. But by Poincar\'{e} inequality and 
   theorem \ref{lpestimatesconstantcoeff}, 
   $$\left\lVert v_{1} - v_{2} \right\rVert_{W^{1,s}} \leq c \left[ \mathcal{C}_{A}(R) + \mathcal{C}_{\Phi}(R) \right] \left\lVert \nabla V_{1} - \nabla V_{2} \right\rVert_{L^{s}}. 
$$ This means for $R$ small enough, $T$ is a contraction and thus has an unique fixed point in $W^{1,s}_{T,flat}.$ The estimate for $p \geq 2$ follows since lower order terms can again be absorbed by interpolation. 
The case $1 < p<2$ follows by duality arguments.
\end{proof}

\section{Applications}
A number of results concerning a wide variety of problems follow from the regularity estimates derived in the last section.\smallskip 

\noindent Before we begin, a few general remarks, valid for 
all the theorems in the Sections \ref{hodge A}. \ref{maxwell A} and \ref{stokes} are in order. 
We state here only the Hessian estimates, but their gradient estimate versions hold as well. 
Also, each of them have corresponding dual versions which can be easily obtained by Hodge duality. Also, the norm of the solution on the right hand side of the 
estimates can be dropped, since uniqueness of the solution is part of the conclusion. But we present the estimates in the given form since in general they can not be improved 
in situations when the solution is only unique up to harmonic fields, as discussed in the remarks following the theorems.    
\subsection{Second order Hodge type systems: A special case}\label{hodge A}
First such problem is the Hodge-type boundary value problem for a second order elliptic system, the prototype being the Poisson problem for the Hodge Laplacian with absolute and 
relative boundary conditions. 
\begin{theorem}\label{second order hodge system tangential}
Let $1 \leq k \leq n-1,$ $r \geq 0$ be integers and $ 0 <  \gamma < 1$ and $1 < p < \infty$ be real numbers. 
Let $\Omega \subset \mathbb{R}^n$ be an open, bounded $C^{r+2},$ respectively $C^{r+2, \gamma},$ set.
  Let  $A \in C^{r+1}(\overline{\Omega} ; L(\Lambda^{k+1},\Lambda^{k+1}))$, respectively 
 $C^{r+1, \gamma}(\overline{\Omega} ; L(\Lambda^{k+1},\Lambda^{k+1})),$ satisfy the Legendre-Hadamard condition. Then the following holds.
\begin{enumerate}
 \item There exists a constant $\rho \in \mathbb{R}$ and an at most countable set 
$\sigma \subset ( -\infty, \rho )$, with no limit points except possibly $- \infty.$
\item The following boundary value problem,
\begin{equation}\label{eigenvalue problem hodge system full regularity}
 \left\lbrace \begin{gathered}
                \delta ( A (x) d\alpha ) + d \delta \alpha  = \sigma_{i} \alpha  \text{ in } \Omega, \\
                 \nu\wedge \alpha = 0 \text{  on } \partial\Omega, \\
                 \nu\wedge \delta\alpha = 0 \text{ on } \partial\Omega.
                \end{gathered} 
                \right. \tag{$EP^{A}_{T}$}
\end{equation}
has non-trivial solutions $\alpha \in W^{r+2,p}(\Omega, \Lambda^{k}),$ respectively $C^{r+2,\gamma}(\overline{\Omega}, \Lambda^{k}),$ 
if and only if  $\sigma_{i} \in \sigma$ and the space of solutions to \eqref{eigenvalue problem hodge system full regularity} is finite-dimensional for 
any $\sigma_{i} \in \sigma.$
\item If $\lambda \notin \sigma$, then 
for any $f \in W^{r,p}(\Omega, \Lambda^{k}) $, respectively $C^{r,\gamma}(\overline{\Omega}, \Lambda^{k}),$ and any $\omega_{0} \in W^{r+2,p}(\Omega, \Lambda^{k}),$ 
respectively $C^{r+2,\gamma}(\overline{\Omega}, \Lambda^{k}),$ there exists a unique 
 solution $\omega \in W^{r+2,p}(\Omega, \Lambda^{k}),$ respectively $C^{r+2,\gamma}(\overline{\Omega}, \Lambda^{k}),$ to the following boundary value problem:
 \begin{equation}\label{bvp hodge elliptic system full regularity}
 \left\lbrace \begin{gathered}
                \delta ( A (x) d\omega )  + d\delta \omega  =  \lambda\omega + f  \text{ in } \Omega, \\
                \nu\wedge \omega = \nu\wedge\omega_{0} \text{  on } \partial\Omega. \\
                \nu\wedge \delta \omega = \nu\wedge\delta\omega_{0} \text{ on } \partial\Omega, 
                \end{gathered} 
                \right. \tag{$P^{A}_{T}$}
\end{equation}
which satisfies the estimate
\begin{align*}
 \left\lVert \omega \right\rVert_{W^{r+2,p}} \leq c \left( \left\lVert \omega\right\rVert_{W^{r,p}} + \left\lVert f\right\rVert_{W^{r,p}} 
 + \left\lVert \omega_{0} \right\rVert_{W^{r+2,p}}\right),
\end{align*}
respectively,  
\begin{align*}
 \left\lVert \omega \right\rVert_{C^{r+2,\gamma}} \leq c \left( \left\lVert \omega \right\rVert_{C^{r,\gamma}} + \left\lVert f \right\rVert_{C^{r,\gamma}} 
 + \left\lVert \omega_{0} \right\rVert_{C^{r+2,\gamma}}\right).
\end{align*}
\end{enumerate}
 \end{theorem}
 \begin{remark}\label{spectrum} (i) The regularity of the eigenform $\alpha$ solving \eqref{eigenvalue problem hodge system full regularity} is limited only by the regularity of the coefficient matrix 
 $A$ and regularity of the boundary $\partial\Omega.$ If both $A$ and $\partial\Omega$ are $C^{\infty},$ then $\alpha \in C^{\infty}(\overline{\Omega}, \Lambda^{k}).$
 
 (ii) Since $A$ satisfies only the Legendre-Hadamard condition, it is possible that there are nonnegative eigenvalues, i.e there exists $\sigma_{i} \geq 0$ in $\sigma.$ However, when $A$ has constant coefficients or satisfies the Legendre condition, since we have the inequality 
 $$\int_{\Omega} \langle A (x) d\alpha , d\alpha \rangle \geq c \int_{\Omega} \left\lvert d\alpha \right\rvert^{2}$$
  for some $c >0,$ it is easy to show that there can not exist a non-trivial solution to \eqref{eigenvalue problem hodge system full regularity} with $\sigma_{i} > 0$ and thus 
  $\sigma \subset (-\infty, 0 ]$ and \eqref{bvp hodge elliptic system full regularity} can be solved for arbitrary $f$ for any $\lambda >0.$ However, $0 \in \sigma$ is still possible. See remark \ref{lambda0tangential}. 
  \end{remark}
 \begin{remark}\label{lambda0tangential} (i) The possibility of solving \eqref{bvp hodge elliptic system full regularity} when $\lambda = 0$ is clearly the most important case. However, if there are non-trivial harmonic fields with 
 vanishing tangential component, i.e $\mathcal{H}_{T}\left( \Omega; \Lambda^{k} \right) \neq \left\lbrace 0 \right\rbrace ,$ then 
 each non-trivial  $h \in \mathcal{H}_{T}\left( \Omega; \Lambda^{k} \right)$ is always a non-trivial solution for \eqref{eigenvalue problem hodge system full regularity} 
 with $\sigma_{i} =0.$ Thus $0 \in \sigma$ and \eqref{bvp hodge elliptic system full regularity} 
 can not be solved for arbitrary $f.$ 
 
 (ii) However, when $A$ has constant coefficients or satisfies the Legendre condition, due to the same arguments as in remark \ref{spectrum}(ii), it can be shown that every non-trivial solution of 
  \eqref{eigenvalue problem hodge system full regularity} with $\sigma_{i} = 0$ must be a harmonic field. Consequently, \eqref{bvp hodge elliptic system full regularity} can be solved for all $f \in 
  \left(\mathcal{H}_{T}\left( \Omega; \Lambda^{k} \right)\right)^{\perp}.$ Moreover, the solution $\omega$ in that case 
  would be unique only up to such harmonic fields. Since harmonic fields are as smooth as $\partial\Omega$ is, the regularity results would remain valid for all such solutions.
  Of course, this extra condition of $f$ would be automatically satisfied if $\mathcal{H}_{T}\left( \Omega; \Lambda^{k} \right) = \left\lbrace 0 \right\rbrace .$ In particular, for 
  contractible domains, \eqref{bvp hodge elliptic system full regularity} with $\lambda = 0$ can be solved for arbitrary $f$ for every $k,$ as soon as 
  $A$ satisfies the Legendre condition or has constant coefficients.  
 \end{remark}

\begin{proof}
 We divide the proof in two steps. \smallskip

\noindent \emph{Step 1 (Existence in $L^{2}$):} 
 We first show existence assuming $f \in L^{p}$ with $p \geq 2.$ Clearly, this is not a restriction if $f \in C^{0,\gamma}.$ 
For a given $\lambda \in \mathbb{R}$, the bilinear operator  
$a_{\lambda}: W_{T}^{1,2}(\Omega;\Lambda^k) \times W_{T}^{1,2}(\Omega;\Lambda^k) \rightarrow \mathbb{R}$ defined by, 
\begin{align*}
 a_{\lambda}(u,v)  =\int_{\Omega}   \langle A (x)d u  , d v \rangle + \int_{\Omega}   \langle \delta u  , \delta v \rangle + \lambda \int_{\Omega} \langle  u , v \rangle ,
\end{align*}
is continuous  and coercive by virtue of the Gaffney and the G\aa{}rding inequality, for any $\lambda $ large enough. Thus Lax-Milgram theorem yields existence for $\lambda$ large.
Since $W_{T}^{1,2}(\Omega;\Lambda^k)$ embeds compactly in $L^{2}(\Omega;\Lambda^k),$ Fredholm theory holds and we conclude the existence of 
the spectrum $\sigma$, the claimed properties of the spectrum, the existence of $\overline{\omega} \in W_{T}^{1,2}(\Omega;\Lambda^k)$ solving 
\begin{align}\label{equationwithf}
 \int_{\Omega}  \langle A (x) d\overline{\omega} , d\theta \rangle + \langle \delta \overline{\omega} , \delta \theta \rangle + 
  \lambda  \int_{\Omega}\langle \overline{\omega} , \theta \rangle 
  +\int_{\Omega} \langle  g, \theta\rangle = 0  , 
\end{align} for any $g \in L^{2}(\Omega, \Lambda^{k})$ when $\lambda \notin \sigma$ and the existence of $\alpha \in W_{T}^{1,2}(\Omega;\Lambda^k)$ solving
\begin{align}\label{eigenvalueequation}
 \int_{\Omega} \langle A (x)d\alpha , d\theta \rangle +  \langle \delta \alpha , \delta \theta \rangle + \sigma_{i}  \int_{\Omega}\langle \alpha , \theta \rangle
   = 0,  
\end{align} for any $\sigma_{i} \in \sigma,$ along with the finite-dimensionality of the eigenspaces.\smallskip

\noindent \emph{Step 2 (Regularity and boundary conditions):} Now theorem \ref{boundary Wrp regularity linear}, respectively theorem \ref{boundary Cralpha regularity linear} in the H\"{o}lder case, 
gives us the desired regularity. The estimates for $1 < p < 2$ also extends the existence theory to the case $1 < p < 2$ by usual density arguments.   
From \eqref{equationwithf}, integrating by parts, we obtain,
$$ \int_{\Omega} \langle \delta  ( A (x) d\overline{\omega} )   + d \delta \overline{\omega} ;  \phi \rangle -  \int_{\partial\Omega} \left( \langle A(x) d\overline{\omega}; \nu\wedge\phi \rangle   
+  \langle \nu\wedge\delta\omega ; \phi \rangle \right)
    = \int_{\Omega} \langle  \lambda\omega + g ; \phi \rangle ,    $$
    for all  $ \phi \in W_{T}^{1,2}(\Omega, \Lambda^{k}).$
 Thus taking  $\phi \in C_{c}^{\infty}(\Omega, \Lambda^{k})$ we have, 
 $$\delta ( A(x) d\overline{\omega} )   + d \delta \overline{\omega} = \lambda\overline{\omega} + g \qquad \text{ in } \Omega. $$
 But this implies that the integral on the boundary vanish separately and thus,
 $$ \int_{\partial\Omega}   \langle \nu\wedge\delta \overline{\omega} ; \phi \rangle  =0 \qquad \text{ for any } \phi \in W_{T}^{1,2}(\Omega; \Lambda^{k}).$$ 
  Extending $\nu$ as a $C^{1}$ function inside $\Omega,$ we see that $ \nu\wedge\delta \overline{\omega}  \in  W_{T}^{1,2}(\Omega, \Lambda^{k})$ 
 and thus, we have,
 $$ \int_{\partial\Omega}   \left\lvert \nu\wedge\delta \overline{\omega}\right\rvert^{2} = 
 \int_{\partial\Omega}   \langle \nu\wedge\delta \overline{\omega} ; \nu\wedge\delta \overline{\omega}  \rangle =  0.$$
This proves that $\nu\wedge\delta \overline{\omega} = 0$ on $\partial\Omega.$ Now taking $g= f + \lambda \omega_{0} -\delta (A (x) d\omega_{0}) - d\delta\omega_{0} $ and setting $\omega = \overline{\omega} + \omega_{0},$ we immediately see that $\omega $ 
    is a solution to \eqref{bvp hodge elliptic system full regularity} with the desired regularity. Similarly, we obtain from \eqref{eigenvalueequation} that $\alpha$ is 
    a solution to \eqref{eigenvalue problem hodge system full regularity}, i.e an eigenform. This finishes the proof. 
\end{proof}

\begin{remark}\label{bvp hodge normal}
(i) Similar arguments on $W^{1,2}_{N}$ for existence and the regularity estimates yield analogous statements for the system  
 \begin{equation}\label{bvp hodge elliptic system full regularity normal}
 \left\lbrace \begin{gathered}
                \delta ( A (x) d\omega )  + d\delta \omega  =  \lambda\omega + f  \text{ in } \Omega, \\
                \nu\lrcorner \omega = \nu\lrcorner\omega_{0} \text{  on } \partial\Omega. \\
                \nu\lrcorner \left( A(x) d\omega \right) = \nu\lrcorner \left( A(x) d\omega_{0} \right) \text{ on } \partial\Omega, 
                \end{gathered} 
                \right. \tag{$P^{A}_{N}$}
\end{equation}
and the corresponding eigenvalue problem when $A$ satisfies the Legendre condition. 

(ii) Due to the stronger ellipticity hypothesis, 
\eqref{bvp hodge elliptic system full regularity normal} can always be solved for arbitrary $f$ with $\lambda >0$ and 
can be solved with $\lambda = 0$ for all $f \in 
  \left(\mathcal{H}_{N}\left( \Omega; \Lambda^{k} \right)\right)^{\perp}.$ If $\mathcal{H}_{N}\left( \Omega; \Lambda^{k} \right) = \lbrace 0 \rbrace,$ \eqref{bvp hodge elliptic system full regularity normal} 
  can be solved for arbitrary $f$ for any $\lambda \geq 0$ and in particular, if $\Omega$ is contractible, this holds true for any $k.$
\end{remark}
 Theorem \ref{second order hodge system tangential} and remark \ref{bvp hodge normal} imply a new proof of the regularity results for the Hodge Laplcian 
 and consequently the Hodge decomposition theorems (see e.g Theorem 6.21 in \cite{CsatoDacKneuss} for the different versions of the theorem).

\subsection{Maxwell operator: A simple case}\label{maxwell A}
\begin{theorem}\label{maxwell linear regularity full}
 Let $1 \leq k \leq n-1,$ $r \geq 0$ be integers and $ 0 < \gamma < 1$ and $1 < p < \infty$ be real numbers. 
 Let $\Omega \subset \mathbb{R}^n$ be an open, bounded $C^{r+2},$ respectively $C^{r+2, \gamma},$ set.
  Let $A \in C^{r+1}(\overline{\Omega} ; L(\Lambda^{k+1},\Lambda^{k+1}))$, respectively 
 $C^{r+1, \gamma}(\overline{\Omega} ; L(\Lambda^{k+1},\Lambda^{k+1})),$ satisfy the Legendre-Hadamard condition. Then the following holds.
\begin{enumerate}
 \item There exists a constant $\rho \in \mathbb{R}$ and an at most countable set 
$\sigma \subset ( -\infty, \rho )$, with no limit points except possibly $- \infty.$
\item The following boundary value problem,
\begin{equation}\label{eigenvalue problem full regularity}
 \left\lbrace \begin{gathered}
                \delta ( A (x) d\alpha )   = \sigma_{i} \alpha  \text{ in } \Omega, \\
                \delta \alpha = 0 \text{ in } \Omega, \\
                \nu\wedge \alpha = 0 \text{  on } \partial\Omega,
                \end{gathered} 
                \right. \tag{$EM^{A}_{T}$}
\end{equation}
has non-trivial solutions $\alpha \in W^{r+2,p}(\Omega, \Lambda^{k}),$ respectively $C^{r+2,\gamma}(\overline{\Omega}, \Lambda^{k}),$ if and only if $\sigma_{i} \in \sigma$ and the space of solutions to \eqref{eigenvalue problem full regularity} is finite-dimensional for any $\sigma_{i} \in \sigma$.
\item If $\lambda \notin \sigma$, then 
for any $f \in W^{r,p}(\Omega, \Lambda^{k}) $, respectively $C^{r,\gamma}(\overline{\Omega}, \Lambda^{k}),$ satisfying $\delta f = 0$ in the sense of distributions in $\Omega,$ there exists a unique
 solution $\omega \in W^{r+2,p}(\Omega, \Lambda^{k}),$ respectively $C^{r+2,\gamma}(\overline{\Omega}, \Lambda^{k}),$ to the following boundary value problem:
 \begin{equation}\label{bvp maxwell delta 0 full}
 \left\lbrace \begin{gathered}
                \delta ( A (x) d\omega )   = \lambda \omega + f  \text{ in } \Omega, \\
                \delta \omega = 0 \text{ in } \Omega, \\
                \nu\wedge \omega = 0 \text{  on } \partial\Omega.
                \end{gathered} 
                \right. \tag{$PM^{A}_{T}$}
\end{equation}
which satisfies the estimate
\begin{align*}
 \left\lVert \omega \right\rVert_{W^{r+2,p}} \leq c \left( \left\lVert \omega\right\rVert_{W^{r,p}} + \left\lVert f\right\rVert_{W^{r,p}} \right),
\end{align*}
respectively,  
\begin{align*}
 \left\lVert \omega \right\rVert_{C^{r+2,\gamma}} \leq c \left( \left\lVert \omega \right\rVert_{C^{r,\gamma}} + \left\lVert f \right\rVert_{C^{r,\gamma}} \right).
\end{align*}
\end{enumerate}
 \end{theorem}\smallskip
 \begin{remark}\label{lambda0maxwelltangential}
  In this case too, remarks similar to remark \ref{spectrum} and \ref{lambda0tangential} applies. In particular, if $\Omega$ is contractible and $A$ satisfies the Legendre condition or 
  has constant coefficients, then for any $k,$ the problem 
  \eqref{bvp maxwell delta 0 full} can be solved for arbitrary $f$ satisfying $\delta f = 0$ for any $\lambda \geq 0.$ Also, eigenforms are $C^{\infty}$ if both $A$ and $\partial\Omega$ are $C^{\infty}$. 
 \end{remark}
 
 \begin{proof} By theorem \ref{second order hodge system tangential}, we only need to show is that if $\delta f = 0,$ then every solution to 
 \eqref{bvp hodge elliptic system full regularity} with $\lambda \in \mathbb{R}$ gives rise to a solution of \eqref{bvp maxwell delta 0 full} with the same $\lambda$ and that every solution of 
 \eqref{eigenvalue problem hodge system full regularity} for $\sigma_{i} \in \mathbb{R}$ yields a solution of \eqref{eigenvalue problem full regularity} with the 
 same $\sigma_{i}.$  We start with the second. Note that if $\delta\alpha = 0,$ then $\alpha$ itself solves \eqref{eigenvalue problem full regularity}. Taking $\delta $ of the equation, 
 we deduce from \eqref{eigenvalue problem hodge system full regularity}, 
 \begin{equation*}
\left( \delta d + d \delta \right) \left( \delta \alpha \right) =  \delta d  \delta \alpha =  \delta \delta ( A (x) d\alpha )+ \delta d  \delta \alpha = \delta 
\left( \sigma_{i}\alpha \right) = \sigma_{i} \left( \delta \alpha\right) \qquad \text{ in } \Omega, \end{equation*}
and $ \nu\wedge \delta\alpha = 0$ and $\nu\wedge \delta\delta\alpha = 0 $ on $\partial\Omega.$ Thus, unless $\delta\alpha = 0,$ $\sigma_{i}$ is an eigenvalue for the Hodge Laplacian 
on $\left( k-1 \right)$-forms with $ \delta \alpha$ as an eigenform. Now 
if $\sigma_{i} = 0,$ then we see that $\delta\alpha = h \in \mathcal{H}_{T} \left( \Omega; \Lambda^{k-1} \right).$ But no non-zero harmonic field can be coexact. If $\sigma_{i} \neq 0,$ then it is easy to check that
$\bar{\alpha} = \alpha - \frac{1}{\sigma_{i}} d \delta\alpha $ solves \eqref{eigenvalue problem full regularity} and has the same regularity as $\alpha.$ The first one is exactly the same. Once again, taking $\delta$ of \eqref{bvp hodge elliptic system full regularity}  and using the fact that 
$\delta f = 0,$ we see similarly that if $\lambda = 0,$ $\omega$ itself and if $\lambda \neq 0,$
 $\bar{\omega} = \omega - \frac{1}{\lambda}d\delta\omega$ solves \eqref{bvp maxwell delta 0 full} and has the same regularity as $\omega.$ 
 \end{proof}

\noindent By similar arguments, we also have
\begin{theorem}\label{maxwell linear regularity full normal}
Let $1 \leq k \leq n-1,$ $r \geq 0$ be integers and $ 0 < \gamma < 1$ and $1 < p < \infty$ be real numbers. 
Let $\Omega \subset \mathbb{R}^n$ be an open, bounded $C^{r+2},$ respectively $C^{r+2, \gamma},$ set.
Let $A \in C^{r+1}(\overline{\Omega} ; L(\Lambda^{k+1},\Lambda^{k+1}))$, respectively 
 $C^{r+1, \gamma}(\overline{\Omega} ; L(\Lambda^{k+1},\Lambda^{k+1})),$ satisfy the Legendre condition. 
Then the following holds.
\begin{enumerate}
 \item There exists a constant $\rho \in \mathbb{R}$ and an at most countable set 
$\sigma \subset ( -\infty, \rho )$, with no limit points except possibly $- \infty.$
\item The following boundary value problem,
\begin{equation}\label{eigenvalue problem full regularity normal}
 \left\lbrace \begin{gathered}
                \delta ( A (x) d\alpha )   = \sigma_{i} \alpha  \text{ in } \Omega, \\
                \delta \alpha = 0 \text{ in } \Omega, \\
                \nu\lrcorner \alpha = 0 \text{  on } \partial\Omega, \\
                \nu\lrcorner \left( A d\alpha \right) = 0 \text{  on } \partial\Omega,
                \end{gathered} 
                \right. \tag{$EM^{A}_{N}$}
\end{equation}
has non-trivial solutions $\alpha \in W^{r+2,p}(\Omega, \Lambda^{k}),$ respectively $C^{r+2,\gamma}(\overline{\Omega}, \Lambda^{k}),$ if and only if $\sigma_{i} \in \sigma$ 
and the space of solutions to \eqref{eigenvalue problem full regularity normal} is finite-dimensional for any $\sigma_{i} \in \sigma$.
\item If $\lambda \notin \sigma$, then 
for any $f \in W^{r,p}(\Omega, \Lambda^{k}) $, respectively $C^{r,\gamma}(\overline{\Omega}, \Lambda^{k}),$ satisfying $\delta f = 0$  in the sense of distributions 
in $\Omega$ and $ \nu\lrcorner f = 0$ 
on $\partial\Omega,$ there exists a
 solution $\omega \in W^{r+2,p}(\Omega, \Lambda^{k}),$ respectively $C^{r+2,\gamma}(\overline{\Omega}, \Lambda^{k}),$  unique upto harmonic fields, to the following boundary value problem:
 \begin{equation}\label{bvp maxwell delta 0 full normal}
 \left\lbrace \begin{gathered}
                \delta ( A (x) d\omega )   = \lambda \omega + f  \text{ in } \Omega, \\
                \delta \omega = 0 \text{ in } \Omega, \\
                \nu\lrcorner \omega = 0 \text{  on } \partial\Omega, \\
                \nu\lrcorner \left( A d\omega \right) = 0 \text{  on } \partial\Omega.
                \end{gathered} 
                \right. \tag{$PM^{A}_{N}$}
\end{equation}
which satisfies the estimate
\begin{align*}
 \left\lVert \omega \right\rVert_{W^{r+2,p}} \leq c \left\lbrace \left\lVert \omega\right\rVert_{W^{r,p}} + \left\lVert f\right\rVert_{W^{r,p}} \right\rbrace,
\end{align*}
respectively,  
\begin{align*}
 \left\lVert \omega \right\rVert_{C^{r+2,\gamma}} \leq c \left\lbrace \left\lVert \omega \right\rVert_{C^{r,\gamma}} + \left\lVert f \right\rVert_{C^{r,\gamma}} \right\rbrace.
\end{align*}
\end{enumerate}
 \end{theorem}
 
 \begin{remark}\label{lambda0maxwellnormal}
(i)  When $r=0,$ the condition $ \nu\lrcorner f = 0$ 
on $\partial\Omega$ in the Sobolev case is to be interpreted as
$$ \int_{\Omega} \left\langle f; d\psi \right\rangle  = 0 \qquad \text{ for any } \psi \in C^{\infty}\left(\overline{\Omega}; \Lambda^{k-1} \right).$$ Note that when $\psi$ is compactly supported, 
the integral vanishes anyway due to the condition $\delta f = 0.$ This extra condition, which can be easily seen to be a necessary one, is analogous the integral compatibility condition 
for the Neumann boundary value problem for the Laplacian in the scalar case.

(ii) Considerations similar to Remark \ref{bvp hodge normal}(ii) apply here as well for the solvability with $\lambda \geq 0$ for all $f$ satisfying the compatibility conditions 
$\delta f = 0$ in $\Omega$ and $ \nu\lrcorner f = 0$ 
on $\partial\Omega$. \end{remark}

\subsection{Stokes type operator}\label{stokes}
The last results give us regularity results for some stationary Stokes type problem. We just use Hodge-Morrey decomposition 
with vanishing tangential component to write $ f = d\phi + \tilde{f},$ with $\delta\tilde{f} = 0$ in $\Omega.$ Then we solve the boundary value problem (see Theorem 8.16 in \cite{CsatoDacKneuss})
$$ dp =  - d\phi \text{ in } \Omega, \qquad \qquad  p = p_{0} \text{ on } \partial\Omega.$$ This reduces the next theorem to Theorem \ref{maxwell linear regularity full}. 
\begin{theorem}\label{Stokes tangential}
 Let $1 \leq k \leq n-1,$ $r \geq 0$ be integers and $ 0 <  \gamma < 1$ and $1 < p < \infty$ be real numbers. Let $\Omega \subset \mathbb{R}^n$ be an open, bounded $C^{r+2},$ respectively $C^{r+2, \gamma},$ set.
  Let $A \in C^{r+1}(\overline{\Omega} ; L(\Lambda^{k+1},\Lambda^{k+1}))$, respectively 
 $C^{r+1, \gamma}(\overline{\Omega} ; L(\Lambda^{k+1},\Lambda^{k+1})),$ satisfy the Legendre-Hadamard condition.
 Let $\sigma$ be the spectrum, as in theorem \ref{maxwell linear regularity full}.
Then for any $\lambda \notin \sigma$ and for any $f \in W^{r,p}(\Omega, \Lambda^{k}) $, respectively $C^{r,\gamma}(\overline{\Omega}, \Lambda^{k}),$ any 
$p_{0} \in W^{r+1,p}(\Omega, \Lambda^{k-1}) $, respectively $C^{r+1,\gamma}(\overline{\Omega}, \Lambda^{k-1}),$ with $\nu\wedge dp_{0} = 0$ on $\partial\Omega,$
there exists a unique pair $\left( \omega, p \right)$ such that $\omega \in W^{r+2,p}(\Omega, \Lambda^{k}),$ respectively $C^{r+2,\gamma}(\overline{\Omega}, \Lambda^{k}),$ 
$p \in W^{r+1,p}(\Omega, \Lambda^{k-1}),$ respectively $C^{r+1,\gamma}(\overline{\Omega}, \Lambda^{k-1}),$ solve the  following boundary value problem:
 \begin{equation}\label{bvp stokes full}
 \left\lbrace \begin{gathered}
                \delta ( A (x) d\omega ) + dp  = \lambda \omega + f   \text{ in } \Omega, \\
                \delta \omega = 0 \text{ in } \Omega, \\
                \nu\wedge \omega = 0 \text{  on } \partial\Omega \\
                p = p_{0} \text{  on } \partial\Omega. 
                \end{gathered} 
                \right. \tag{$PS_{T}$}
\end{equation}
which satisfies the estimates
\begin{align*}
 \left\lVert \omega \right\rVert_{W^{r+2,p}} \leq c \left( \left\lVert \omega\right\rVert_{W^{r,p}} + \left\lVert f\right\rVert_{W^{r,p}} \right)  \\ 
 \intertext{ and } 
 \left\lVert p \right\rVert_{W^{r+1,p}} \leq c \left( \left\lVert f\right\rVert_{W^{r,p}} + \left\lVert p_{0}\right\rVert_{W^{r+1,p}} \right), 
\end{align*}
respectively,  
\begin{align*}
 \left\lVert \omega \right\rVert_{C^{r+2,\gamma}} \leq c \left( \left\lVert \omega \right\rVert_{C^{r,\gamma}} + \left\lVert f \right\rVert_{C^{r,\gamma}} \right) \\ 
 \intertext{ and } \quad \left\lVert p \right\rVert_{C^{r+1,\gamma}} \leq c\left( \left\lVert f \right\rVert_{C^{r,\gamma}}+ \left\lVert p_{0} \right\rVert_{C^{r+1,\gamma}} \right) .
\end{align*}
\end{theorem}
\begin{remark}
 Considerations similar to remark \ref{lambda0maxwelltangential} apply here as well. 
\end{remark}

By using Hodge-Morrey decomposition with vanishing normal component to write  $ f = -dp + \tilde{f},$ with $\delta\tilde{f} = 0$ in $\Omega$ and 
$\nu\lrcorner\tilde{f} = 0$ on $\partial\Omega,$ we also have the following theorem.
\begin{theorem}\label{Stokes normal}
 Let $1 \leq k \leq n-1,$ $r \geq 0$ be integers and $ 0 < \gamma < 1$ and $1 < p < \infty$ be real numbers. 
 Let $\Omega \subset \mathbb{R}^n$ be an open, bounded $C^{r+2},$ respectively $C^{r+2, \gamma},$ set.
 Let $A \in C^{r+1}(\overline{\Omega} ; L(\Lambda^{k+1},\Lambda^{k+1}))$, respectively 
 $C^{r+1, \gamma}(\overline{\Omega} ; L(\Lambda^{k+1},\Lambda^{k+1})),$ satisfy the Legendre condition.
 Let $\sigma$ be the spectrum, as given by theorem \ref{maxwell linear regularity full normal}.
Then for any $\lambda \notin \sigma$, for any $f \in W^{r,p}(\Omega, \Lambda^{k}) $, respectively $C^{r,\gamma}(\overline{\Omega}, \Lambda^{k}),$ there exists a unique pair
$\left( \omega, p \right)$ such that $\omega \in W^{r+2,p}(\Omega, \Lambda^{k}),$ respectively $C^{r+2,\gamma}(\overline{\Omega}, \Lambda^{k}),$ 
$p \in W^{r+1,p}(\Omega, \Lambda^{k-1}),$ respectively $C^{r+1,\gamma}(\overline{\Omega}, \Lambda^{k-1}),$ solve the  following boundary value problem:
 \begin{equation}\label{bvp stokes normal}
 \left\lbrace \begin{gathered}
                \delta ( A (x) d\omega ) + dp  = \lambda \omega + f  \text{ in } \Omega, \\
                \delta \omega = 0 \text{ in } \Omega, \\
                \nu\lrcorner \omega = 0 \text{  on } \partial\Omega, \\
                \nu\lrcorner \left( A(x) d \omega \right)  = 0 \text{  on } \partial\Omega.
                \end{gathered} 
                \right. \tag{$PS_{N}$}
\end{equation}
and satisfies the estimate
\begin{align*}
 \left\lVert \omega \right\rVert_{W^{r+2,p}} \leq c \left( \left\lVert \omega\right\rVert_{W^{r,p}} + \left\lVert f\right\rVert_{W^{r,p}} \right) \quad \text{ and } \quad 
 \left\lVert p \right\rVert_{W^{r+1,p}} \leq c \left\lVert f\right\rVert_{W^{r,p}}, 
\end{align*}
respectively,  
\begin{align*}
 \left\lVert \omega \right\rVert_{C^{r+2,\gamma}} \leq c \left( \left\lVert \omega \right\rVert_{C^{r,\gamma}} + \left\lVert f \right\rVert_{C^{r,\gamma}} \right) 
 \quad \text{ and } \quad \left\lVert p \right\rVert_{C^{r+1,\gamma}} \leq c\left\lVert f \right\rVert_{C^{r,\gamma}}.
\end{align*}
 \end{theorem}
 \begin{remark}
 Considerations similar to remark \ref{lambda0maxwellnormal} apply here as well. 
\end{remark}
\begin{remark}
For $k=1$ and $n=3,$ the above two problems reduce to 
\begin{align*}
  \operatorname*{curl} ( A (x) \operatorname*{curl}u ) + \nabla p  &= \lambda u + f   \text{ in } \Omega, \notag \\
                \operatorname*{div} u &= 0 \text{ in } \Omega,  \notag \\
 \text{ and either } \quad \left. \begin{gathered}
                \nu\times u = 0 ,  \\
                p = p_{0},                
                \end{gathered} 
                \right. \qquad &\text{ or } \qquad \left. \begin{gathered}
                \nu\cdot u = 0 , \\
                \nu \times \left( A(x) \operatorname*{curl}u \right)  = 0,
                \end{gathered} 
                \right. \qquad \text{  on } \partial\Omega.
\end{align*}
These are usually called the \emph{vorticity-velocity-pressure formulation of the incompressible stationary Stokes problem}. See, for example, 
\cite{BeiraoBerselliNavierStokesstressfreebc}, \cite{ConcaPironneauNavierStokes}, \cite{DuboisStokes}. 
\end{remark}
\subsection{A non-elliptic Dirichlet problem}\label{maxwell dirichlet}
\begin{theorem}\label{dirichlet problem arbitrary boundary data}
Let $1 \leq k \leq n-1,$ $r \geq 0$ be integers and $ 0 < \gamma < 1$ and $1 < p < \infty$ be real numbers. 
Let $\Omega \subset \mathbb{R}^n$ be an open, bounded $C^{r+2},$ respectively $C^{r+2, \gamma},$ set.
 Let $A \in C^{r+1}(\overline{\Omega} ; L(\Lambda^{k+1},\Lambda^{k+1}))$, respectively 
 $C^{r+1, \gamma}(\overline{\Omega} ; L(\Lambda^{k+1},\Lambda^{k+1})),$ satisfy either the Legendre condition or satisfies the Legendre-Hadamard condition and has constant 
 coefficients. 
Then for any $\omega_{0} \in W^{r+2,p}\left(\Omega, \Lambda^{k}\right),$ 
respectively $C^{r+2,\gamma}(\overline{\Omega}, \Lambda^{k}),$ and for 
any $f \in W^{r,p}\left(\Omega, \Lambda^{k}\right)\cap \left( \mathcal{H}_{T}\left(\Omega, \Lambda^{k}\right) \right)^{\perp}, $ 
respectively $C^{r,\gamma}(\overline{\Omega}, \Lambda^{k})\cap \left( \mathcal{H}_{T}\left(\Omega, \Lambda^{k}\right) \right)^{\perp},$ 
such that $\delta f = 0$ in the sense of distributions, there exists a solution 
$\omega \in W^{r+p,2}\left(\Omega, \Lambda^{k}\right),$ respectively $C^{r+2,\gamma}(\overline{\Omega}, \Lambda^{k}),$ to the following
boundary value problem, 
 \begin{equation}\label{bvp maxwell dirichlet}
 \left\lbrace \begin{gathered}
                \delta ( A (x) d\omega )  = f   \text{ in } \Omega, \\
                \omega = \omega_{0}  \text{  on } \partial\Omega,
                \end{gathered} 
                \right. \tag{$\mathcal{P}_{D}$}
\end{equation}
which satisfies the estimate 
\begin{align*}
 \left\lVert \omega \right\rVert_{W^{r+2,p}} \leq c \left(  \left\lVert f\right\rVert_{W^{r,p}}  + \left\lVert \omega_{0}\right\rVert_{W^{r+2,p}}\right), 
\end{align*}
respectively,  
\begin{align*}
 \left\lVert \omega \right\rVert_{C^{r+2,\gamma}} \leq c \left( \left\lVert f \right\rVert_{C^{r,\gamma}} + \left\lVert \omega_{0} \right\rVert_{C^{r+2,\gamma}} \right).
\end{align*}
\end{theorem}
\begin{remark} (i) Note that there is no uniqueness and the claimed estimate is not an apriori estimate. The theorem only says that it is possible to find a solution which enjoys the claimed regularity
 and satisfies these 
 estimates. All solutions of the equation need neither satisfy such estimates nor have the claimed regularity. Indeed, for any 
 $\beta \in W_{0}^{1,2}\left(\Omega, \Lambda^{k-1}\right),$ adding $d\beta$ to any distributional solution of \eqref{bvp maxwell dirichlet} yields another distributional solution 
  which can be merely $L^{2}.$ \smallskip
 
 (ii) Once again if $\mathcal{H}_{T}\left(\Omega, \Lambda^{k}\right) = \lbrace 0 \rbrace,$ $f$ has no restrictions other than $\delta f = 0$ and in particular for 
 contractible domains, this happens for all $k$. 
 \end{remark}
\begin{proof}
 The hypothesis on $A$ implies that by theorem \ref{bvp maxwell delta 0 full}, we can find 
 an unique solution $\overline{\omega},$ with the expected regularity, to the system 
 \begin{equation*}
 \left\lbrace \begin{gathered}
                \delta ( A (x) d\overline{\omega} )   = f  - \delta  (  A (x) d\omega_{0} )  \text{ in } \Omega, \\
                \delta \overline{\omega} = 0 \text{ in } \Omega, \\
                \nu\wedge \overline{\omega} = 0 \text{  on } \partial\Omega,
                \end{gathered} 
                \right. 
\end{equation*}
Now since $\nu\wedge (- \overline{\omega}) = 0$ on $\partial\Omega$, we can find $v \in W^{r+3,2}(\Omega, \Lambda^{k-1})$ (see lemma 8.11 in \cite{CsatoDacKneuss})
 such that $dv = - \overline{\omega}$ on $\partial\Omega$. Setting $\omega = \omega_{0}+\overline{\omega} +dv$, we see that $\omega$ solves \eqref{bvp maxwell dirichlet}. 
\end{proof}

\subsection{First order div-curl type systems and the Gaffney inequalities}
\begin{theorem}\label{linear system assymmetric d-delta}
 Let $1 \leq k \leq n-1$  and $r \geq 0$ be integers and $ 0 <  \gamma < 1$ and $1 < p < \infty$ be real numbers. Let $\Omega \subset \mathbb{R}^n$ be 
 an open, bounded $C^{r+2},$ respectively $C^{r+2, \gamma},$ set. 
 
 \noindent Let  $A,B \in C^{r+1}(\overline{\Omega} ; L(\Lambda^{k},\Lambda^{k})),$ respectively $C^{r+1,\gamma} (\overline{\Omega} ; L(\Lambda^{k},\Lambda^{k})),$  satisfy the 
 Legendre condition. Let $\omega_{0} \in W^{r+1,p}(\Omega, \Lambda^{k})$, respectively $C^{r+1,\gamma}(\overline{\Omega} ; \Lambda^{k}),$ 
 $f \in W^{r,p}(\Omega, \Lambda^{k+1}),$ respectively $C^{r,\gamma}(\overline{\Omega} ; \Lambda^{k+1})$ and $g \in W^{r,p}(\Omega, \Lambda^{k-1}),$
  respectively $C^{r,\gamma}(\overline{\Omega} ; \Lambda^{k-1}).$
 Then the following hold true.\smallskip
 
\noindent \textbf{(i)} Suppose $f$ and $g$  satisfy $df = 0$, $\delta g = 0$ in $\Omega$ and 
 $ \nu\wedge d\omega_{0} = \nu\wedge f $ on $\partial\Omega,$
  and for every $\chi \in \mathcal{H}_T(\Omega;\Lambda^{k+1})$ and $\psi \in \mathcal{H}_T(\Omega;\Lambda^{k-1})$,
         \begin{equation*}
         \int_{\Omega} \langle f ; \chi \rangle - \int_{\partial\Omega} \langle \nu \wedge \omega_0 ; \chi \rangle = 0 
         \text{ and } \int_{\Omega} \langle g ; \psi \rangle = 0. 
         \end{equation*}
 Then there exists a solution $\omega \in W^{ r+1, p}(\Omega, \Lambda^{k}),$ respectively $C^{r+1,\gamma}(\overline{\Omega} ; \Lambda^{k}),$ to the following boundary value problem, 
 \begin{equation} \label{problemddeltalinear}
   \left\lbrace \begin{aligned}
                d(A(x)\omega) = f  \quad &\text{and} \quad  \delta (B(x) \omega) = g &&\text{ in } \Omega, \\
                \nu\wedge A(x)\omega &= \nu\wedge\omega_0 &&\text{  on } \partial\Omega,
                \end{aligned} 
                \right. \tag{$\mathcal{P}_{T}$}
                \end{equation}
satisfying the estimates 
\begin{align*}
 \left\lVert \omega \right\rVert_{W^{r+1,p}} \leq c \left(  \left\lVert  \omega \right\rVert_{L^{p}} + \left\lVert f\right\rVert_{W^{r,p}}  + \left\lVert g\right\rVert_{W^{r,p}} + \left\lVert \omega_{0}\right\rVert_{W^{r+1,p}}\right), 
\end{align*}
respectively,  
\begin{align*}
 \left\lVert \omega \right\rVert_{C^{r+1,\gamma}} \leq c \left( \left\lVert  \omega \right\rVert_{C^{0,\gamma}} + \left\lVert f \right\rVert_{C^{r,\gamma}} + \left\lVert g \right\rVert_{C^{r,\gamma}} 
 + \left\lVert \omega_{0} \right\rVert_{C^{r+1,\gamma}} \right).
\end{align*}
Furthermore, these estimates are apriori estimates, i.e any weak solution $\omega \in L^{p}(\Omega, \Lambda^{k}),$ respectively $C^{0,\gamma}(\overline{\Omega} ; \Lambda^{k}),$
satisfying \eqref{problemddeltalinear} in the sense of distributions with the assumed conditions on $f,$ $g$ and $\omega_{0}$ are actually $W^{ r+1, p}(\Omega, \Lambda^{k}),$ respectively $C^{r+1,\gamma}(\overline{\Omega} ; \Lambda^{k}),$
 and the respective estimates hold.\smallskip
                
\noindent\textbf{(ii)} Suppose $f$ and $g$  satisfy $df = 0$, $\delta g = 0$ in $\Omega$ and $ \nu\lrcorner g  = \nu\lrcorner \delta\omega_{0}$ on $\partial\Omega,  $
and for every $\chi \in \mathcal{H}_N(\Omega;\Lambda^{k-1})$ and $\psi \in \mathcal{H}_N(\Omega;\Lambda^{k+1})$,
         \begin{equation*}
         \int_{\Omega} \langle g ; \chi \rangle - \int_{\partial\Omega} \langle \nu \lrcorner \omega_0 ; \chi \rangle = 0 
         \text{ and } \int_{\Omega} \langle f ; \psi \rangle = 0. 
         \end{equation*}
Then there exists a solution $\omega \in W^{ r+1, p}(\Omega, \Lambda^{k}),$ respectively $C^{r+1,\gamma}(\overline{\Omega} ; \Lambda^{k}),$ to the following boundary 
          value problem,
\begin{equation} \label{problemddeltalinearnormal}
   \left\lbrace \begin{aligned}
                d(A(x)\omega) = f  \quad &\text{and} \quad  \delta (B(x) \omega) = g &&\text{ in } \Omega, \\
                \nu\lrcorner B(x)\omega &= \nu\lrcorner\omega_0 &&\text{  on } \partial\Omega,
                \end{aligned} 
                \right. \tag{$\mathcal{P}_{N}$}
                \end{equation}
                satisfying the estimates 
\begin{align*}
 \left\lVert \omega \right\rVert_{W^{r+1,p}} \leq c \left( \left\lVert  \omega \right\rVert_{L^{p}} +  \left\lVert f\right\rVert_{W^{r,p}}  + \left\lVert g\right\rVert_{W^{r,p}} + \left\lVert \omega_{0}\right\rVert_{W^{r+1,p}}\right), 
\end{align*}
respectively,  
\begin{align*}
 \left\lVert \omega \right\rVert_{C^{r+1,\gamma}} \leq c \left( \left\lVert  \omega \right\rVert_{C^{0,\gamma}} + \left\lVert f \right\rVert_{C^{r,\gamma}} + \left\lVert g \right\rVert_{C^{r,\gamma}} 
 + \left\lVert \omega_{0} \right\rVert_{C^{r+1,\gamma}} \right).
\end{align*}
Furthermore, these estimates are apriori estimates, i.e any weak solution $\omega \in L^{p}(\Omega, \Lambda^{k}),$ respectively $C^{0,\gamma}(\overline{\Omega} ; \Lambda^{k}),$
satisfying \eqref{problemddeltalinearnormal} in the sense of distributions with the conditions on $f,$ $g$ and $\omega_{0}$ are actually $W^{ r+1, p}(\Omega, \Lambda^{k}),$ respectively $C^{r+1,\gamma}(\overline{\Omega} ; \Lambda^{k}),$
 and the respective estimates hold.
\end{theorem}
\begin{remark}\label{ddeltauniquenessremark}(i) When $r=0,$ the conditions $df = 0$, $\delta g = 0$ are understood in the sense of distributions. The conditions 
$\nu\wedge d\omega_{0} = \nu\wedge f \text{ on } \partial\Omega $  and $ \nu\lrcorner \delta\omega_{0} = \nu\lrcorner g 
\text{ on } \partial\Omega $ are well defined in the H\"{o}lder case and are to be interpreted in the weak sense in the Sobolev case 
(see remark 7.3(iii) in \cite{CsatoDacKneuss}). 

(ii) The conditions on $f,g$ and $\omega_{0}$ regarding the harmonic fields are automatically satisfied if there are no such harmonic field. In particular, 
for contractible domains, those conditions present no restriction for both the problems for any $k.$  

(iii) Note that although the estimates are apriori estimates, the solution need not be unique in general and thus, unlike the theorems in Section 
\ref{hodge A}. \ref{maxwell A} and \ref{stokes}, the norm of $\omega$ can not be dropped from the right hand side of the estimates. In fact, for every nontrivial 
$h \in  \mathcal{H}_T(\Omega;\Lambda^{k}),$ $\omega = A^{-1}\left( d\alpha + h \right)$ is a nontrivial solution of  \eqref{problemddeltalinear} with zero data,  where 
$\alpha$ is a solution to 
\begin{align*}
 \left\lbrace \begin{aligned}\delta ( BA^{-1} d \alpha ) &=  - \delta ( BA^{-1}h ) &&\text{ in } \Omega, \\
 \delta \alpha &= 0 &&\text{ in } \Omega, \\
 \nu \wedge \alpha &= 0 &&\text{ on  } \partial \Omega.\end{aligned}\right. \end{align*}
 Conversely, every non-trivial solution to \eqref{problemddeltalinear} with zero data is of this form and thus \eqref{problemddeltalinear} has a unique solution if 
 $\mathcal{H}_T(\Omega;\Lambda^{k}) = \lbrace 0 \rbrace. $     
Similarly, for every nontrivial 
$h \in  \mathcal{H}_N(\Omega;\Lambda^{k}),$ $\omega = B^{-1}\left( \delta\alpha + h \right)$ is a nontrivial solution of  \eqref{problemddeltalinearnormal} with zero data,  where 
$\alpha$ is a solution to 
\begin{align*}
 \left\lbrace \begin{aligned} d ( AB^{-1} \delta \alpha ) &=  - d ( AB^{-1}h ) &&\text{ in } \Omega, \\
 d \alpha &= 0 &&\text{ in } \Omega, \\
 \nu \lrcorner \alpha &= 0 &&\text{ on  } \partial \Omega.\end{aligned}\right. \end{align*} Also, every nontrivial solution to \eqref{problemddeltalinearnormal} with zero 
 data is of this form and thus if  $\mathcal{H}_N(\Omega;\Lambda^{k}) = \lbrace 0 \rbrace, $  \eqref{problemddeltalinearnormal} has a unique solution. In particular, 
 for contractible domains, we have uniqueness of solutions for both \eqref{problemddeltalinear} and \eqref{problemddeltalinearnormal} for any $k.$
\end{remark}
\begin{proof}
We prove only part (i) and the Sobolev case and assume that $2 \leq k \leq n-1.$ The H\"{o}lder case is similar and the case $k=1$ is much easier. Part (ii) follows analogously using the dual versions.  
The hypotheses on $f$ imply (see Theorem 8.16 in \cite{CsatoDacKneuss}) that
 there exists $F \in W^{r+1,p}(\Omega, \Lambda^{k})$ such that 
\begin{align*}
 \left\lbrace \begin{aligned}
               dF &= f  \qquad \text{ in } \Omega ,\\
 F &= \omega_{0} \qquad \text{ on } \partial\Omega .
              \end{aligned}\right.\end{align*}
Note that solvability of this problem is a consequence of the Hodge-Morrey decomposition which in turn follows from the results in Section \ref{hodge A}. 
Now, we find a solution $\alpha \in W^{r+2,p}(\Omega, \Lambda^{k-1})$ such that 
 \begin{align*}
 \left\lbrace \begin{aligned}\delta ( BA^{-1} d \alpha ) &= g - \delta ( BA^{-1} F ) &&\text{ in } \Omega, \\
 \delta \alpha &= 0 &&\text{ in } \Omega, \\
 \nu \wedge \alpha &= 0 &&\text{ on  } \partial \Omega.\end{aligned}\right. \end{align*}
Now setting $$ \omega = A^{-1} (d\alpha + F),$$ we easily verify $\omega$ solves \eqref{problemddeltalinear}.
The estimates are actually apriori estimates, since the argument is essentially reversible.  \end{proof}

\paragraph*{} With the help of the previous results, we can deduce new Gaffney type inequalities (see \cite{FriedrichsGaffney}, \cite{GaffneyHarmonicoperator}, \cite{GaffneyHarmonicintegrals}). 
\begin{theorem}[Gaffney type inequality]\label{newgaffney}
Let $1 \leq k \leq n-1$ be an integer and 
$1 < p < \infty$ and $0 < \gamma < 1$ be real numbers. Let  $\Omega \subset \mathbb{R}^{n}$ be open, bounded and $C^{2},$ respectively $C^{2,\gamma}.$ Let $B \in C^{1}(\overline{\Omega}; L(\Lambda^{k};\Lambda^{k}),$ respectively 
$C^{1,\gamma}(\overline{\Omega}; L(\Lambda^{k};\Lambda^{k}),$  satisfy the Legendre condition.

\begin{itemize}
 \item[(i)]  Let $u \in L^{p}(\Omega;\Lambda^{k}),$ respectively $u \in C^{0,\gamma}(\overline{\Omega}; \Lambda^{k}),$
$ du \in L^{p}(\Omega;\Lambda^{k+1}),$ respectively $du \in C^{0,\gamma}(\overline{\Omega}; \Lambda^{k+1}),$ $\delta(Bu) 
\in L^{p}(\Omega;\Lambda^{k-1}),$ respectively $\delta(Bu) 
\in C^{0,\gamma}(\overline{\Omega}; \Lambda^{k-1}).$ Suppose either $\nu\wedge u =0$ on $\partial\Omega$ or $\nu\lrcorner \left(  B(x)u \right) = 0 $ on $\partial\Omega.$
Then $u \in W^{1,p}(\Omega; \Lambda^{k}), $ respectively $u \in C^{1,\gamma}(\overline{\Omega}; \Lambda^{k})$ and there exists a constant 
$C_{p} = C( \Omega, B, p ) > 0,$ respectively $C_{\gamma} = C( \Omega, B, \gamma ) > 0,$ such that 
\begin{equation*}
   \lVert  u \rVert_{W^{1,p}} \leq C_{p} \left( \lVert du \rVert_{L^{p}(\Omega;\Lambda^{k+1})}  
 + \lVert \delta (Bu) \rVert_{L^{p}(\Omega;\Lambda^{k-1})} + \lVert  u \rVert_{L^{p}(\Omega;\Lambda^{k})}  \right) ,
  \end{equation*}
and respectively
\begin{equation*}
   \lVert u \rVert_{C^{1,\gamma}} \leq C_{\gamma} \left( \lVert du \rVert_{C^{0,\gamma}(\overline{\Omega};\Lambda^{k+1})}  
 + \lVert \delta (Bu) \rVert_{C^{0,\gamma}(\overline{\Omega};\Lambda^{k-1})} + \lVert  u \rVert_{C^{0,\gamma}(\overline{\Omega};\Lambda^{k})}  \right) .
  \end{equation*}
  \item[(ii)] Let $u \in L^{p}(\Omega;\Lambda^{k}),$ respectively $u \in C^{0,\gamma}(\overline{\Omega}; \Lambda^{k}),$
$ d\left( Bu \right) \in L^{p}(\Omega;\Lambda^{k+1}),$ respectively $d \left( Bu \right) \in C^{0,\gamma}(\overline{\Omega}; \Lambda^{k+1}),$ $\delta u 
\in L^{p}(\Omega;\Lambda^{k-1}),$ respectively $\delta u 
\in C^{0,\gamma}(\overline{\Omega}; \Lambda^{k-1}).$ Suppose either $\nu\wedge \left( B(x) u \right) =0$ on $\partial\Omega$ or $\nu\lrcorner u  = 0 $ on $\partial\Omega.$
Then $u \in W^{1,p}(\Omega; \Lambda^{k}), $ respectively $u \in C^{1,\gamma}(\overline{\Omega}; \Lambda^{k})$ and there exists a constant 
$C_{p} = C( \Omega, B, \gamma_{0}, p ) > 0,$ respectively $C_{\gamma} = C( \Omega, B,  \gamma ) > 0,$ such that 
\begin{equation*}
   \lVert  u \rVert_{W^{1,p}} \leq C_{p} \left( \lVert d \left( B u \right) \rVert_{L^{p}(\Omega;\Lambda^{k+1})}  
 + \lVert \delta u \rVert_{L^{p}(\Omega;\Lambda^{k-1})} + \lVert  u \rVert_{L^{p}(\Omega;\Lambda^{k})}  \right) ,
  \end{equation*}
and respectively
\begin{equation*}
   \lVert u \rVert_{C^{1,\gamma}} \leq C_{\gamma} \left( \lVert d \left( B u \right) \rVert_{C^{0,\gamma}(\overline{\Omega};\Lambda^{k+1})}  
 + \lVert \delta u \rVert_{C^{0,\gamma}(\overline{\Omega};\Lambda^{k-1})} + \lVert  u \rVert_{C^{0,\gamma}(\overline{\Omega};\Lambda^{k})}  \right) .
  \end{equation*}
\end{itemize}
\end{theorem}
\begin{remark}\label{newgaffneycontractible}(i) Uniqueness considerations (see remark \ref{ddeltauniquenessremark}(iii)) imply that 
if $\mathcal{H}_{T}\left( \Omega; \Lambda^{k} \right) = \lbrace 0 \rbrace,$ then the term containing $u$ can be dropped in the estimates in part (i) 
 for the boundary condition $\nu\wedge u = 0$ on $\partial\Omega.$ Similarly, if $\mathcal{H}_{N}\left( \Omega; \Lambda^{k} \right) = \lbrace 0 \rbrace,$ then the same 
 can be done in the estimates in part (ii) 
 for the boundary condition $\nu\lrcorner u = 0$ on $\partial\Omega.$ 
 
 (ii) In particular for $p=2,$ if $\Omega$ is open, bounded, $C^{2}$ and contractible and $B \in C^{1}(\overline{\Omega}; L(\Lambda^{k};\Lambda^{k})$  satisfy 
 the Legendre condition, then for any $1 \leq k \leq n-1,$ we have the following inequalities:
 \begin{align}
 \lVert \nabla u \rVert_{L^{2}}^{2} &\leq C \left( \lVert du \rVert_{L^{2}}^{2}  
 + \lVert \delta (Bu) \rVert_{L^{2}}^{2}  \right)  &&\text{ for all } u \in W^{1,2}_{T}(\Omega;\Lambda^{k}),\label{newgaffneytangent}\\
 \intertext{ and }
  \lVert  \nabla u \rVert_{L^{2}}^{2} &\leq C \left( \lVert d \left( B u \right) \rVert_{L^{2}}^{2} + \lVert \delta u \rVert_{L^{2}}^{2}  \right) 
  &&\text{ for all } u \in W^{1,2}_{N}(\Omega;\Lambda^{k}). \label{newgaffneynormal}
 \end{align}
\end{remark}
We shall use these two inequalities crucially in the next section.

\section{Second order Hodge type systems revisited} 
With the help of the new Gaffney type inequalities deduced in the last section, it is now possible to show that the regularity results extend to the general systems as well. The regularity estimates follow essentially the same way 
as in Section \ref{boundary estimates}, the only difference being that we shall now use the new Gaffney type inequalities \eqref{newgaffneytangent}
and \eqref{newgaffneynormal} in place of the usual Gaffney inequality.  

\subsection{Regularity estimates in the tangential case}\label{regularitydiscssiongeneraltan}
As in section \ref{schauderandlp}, using lemma \ref{flattening and freezing coefficients}, proving the regularity estimates for system 
\eqref{bvp hodge elliptic system general full regularity} comes down to proving 
 the estimates for $u \in W_{T}^{1,2}(B_{R}^{+};\Lambda^{k})$ satisfying 
\begin{align*}
 \int_{B_{R}^{+}} \langle \bar{A}(du) ; d\psi \rangle + \int_{B_{R}^{+}} \langle \delta \left( \bar{B} u \right) ; \delta \left( \bar{B}\psi \right)  \rangle 
 +  \int_{B_{R}^{+}} \langle \mathcal{P} ;   \psi & \rangle 
  -\int_{B_{R}^{+}} \langle \mathcal{Q} ; \nabla\psi \rangle \notag \\&+ \int_{B_{R}^{+}} \langle \mathrm{S}\nabla u ; \nabla\psi \rangle= 0,
   \end{align*}
for all $\psi \in W_{T}^{1,2}(B_{R}^{+};\Lambda^{k}),$ where $\bar{A}$ satisfies the Legendre-Hadamard condition and $\bar{B}$ satisfies a Legendre condition, 
$\mathcal{P} = \widetilde{f} + \mathrm{P}u + \mathrm{R} \nabla u$  and $\mathcal{Q} = \widetilde{F}  -  \mathrm{Q}u ,$ with $\widetilde{f}, \widetilde{F}, \mathrm{P}, \mathrm{Q}, 
\mathrm{R}, \mathrm{S}$ are as in lemma \ref{flattening and freezing coefficients}, 
for $R>0$ suitably small. 

\par Note that by the G\aa{}rding inequality and the Gaffney type 
inequality \eqref{newgaffneytangent}, we have, for any $v \in W_{T}^{1,2}(B_{R}^{+};\Lambda^{k}),$ 
 $$\int_{B_{R}^{+}} \langle \bar{A}(dv) ; dv \rangle + \int_{B_{R}^{+}} \left\lvert \delta \left( \bar{B} v \right) \right\rvert^{2} 
 \geq c\left(\int_{B_{R}^{+}}  \lvert dv \rvert^{2} + \int_{B_{R}^{+}} \lvert \delta (\bar{B}v) \rvert^{2} \right) \geq c  \int_{B_{R}^{+}} \lvert \nabla v \rvert^{2} .$$ 
 Using this, the Caccioppoli inequality and all the regularity estimates carry over to this setting as soon as we show that the $L^{2}$ norm of the $D_{nn}$ derivative can be estimated from the equation. To show this, we 
 define the linear map $\widetilde{A}:\mathbb{R}^{\tbinom{n}{k}\times n} \rightarrow \mathbb{R}^{\tbinom{n}{k}\times n}$ by the pointwise algebraic identities
$$\langle \widetilde{A} \left( a_{1} \otimes b_{1} \right) ;  a_{2} \otimes b_{2} \rangle = \langle A\left( a_{1} \wedge b_{1} \right) ; a_{2} \wedge b_{2} \rangle 
+ \langle a_{1} \lrcorner B b_{1} ; a_{2} \lrcorner B b_{2}  \rangle,  $$    for every $ a_{1},a_{2} \in \Lambda^{1}$, viewed also as vectors in $\mathbb{R}^{n}$  and $b_{1}, b_{2} \in \Lambda^{k},$
viewed also as vectors in $\mathbb{R}^{\tbinom{n}{k}}$. Note that this implies, for every $u, \psi \in  W_{T, flat}^{1,2}(B_{R}^{+} ; \Lambda^{k}),$ we have, 
  \begin{align*}
   \int_{B_{R}^{+}} \langle A(du) ; d\psi \rangle + \int_{B_{R}^{+}} \langle \delta \left( B u \right) ; \delta \left( B \psi \right)  \rangle
   = \int_{B_{R}^{+}} \langle \widetilde{A}\nabla u ; \nabla \psi \rangle. 
  \end{align*}
We also define the maps $\widetilde{A}^{pq}:\mathbb{R}^{\tbinom{n}{k}} \rightarrow \mathbb{R}^{\tbinom{n}{k}} $ by \eqref{Apqdefinition} as before. Now we just need to show that $\widetilde{A}^{nn}$ is invertible. Indeed, by virtue of the identity  
 $ \xi = e_{n}\lrcorner \left( e_{n} \wedge \xi \right) + e_{n}\wedge \left(e_{n} \lrcorner \xi \right), $ we have, 
 for every $\xi \in \mathbb{R}^{\tbinom{n}{k}},$
 \begin{align*}
  \langle \widetilde{A}^{nn} \xi ; \xi \rangle &= \langle \widetilde{A} (e_{n} \otimes \xi) ; e_{n} \otimes \xi \rangle 
  = \langle A\left( e_{n} \wedge \xi \right) ; e_{n} \wedge \xi \rangle 
+ \langle e_{n} \lrcorner B\xi ; e_{n} \lrcorner B\xi \rangle \\&\geq \gamma_{A}\lvert e_{n} \wedge \xi \rvert^{2} 
+ \left\lvert e_{n} \lrcorner B\xi \right\rvert^{2} 
\geq c_{1} \left\lvert \xi \right\rvert^{2}, 
 \end{align*}
 where $c_{1}>0$ is the constant given by lemma \ref{ellipticitylemma}.
 
\subsection{Regularity estimates in the normal case}\label{regularitydiscssiongeneralnormal}
The case of the normal type boundary condition is slightly trickier. For a given $B \in C^{r+2}(\overline{\Omega} ; L(\Lambda^{k},\Lambda^{k}))$, respectively 
 $C^{r+2, \gamma}(\overline{\Omega} ; L(\Lambda^{k},\Lambda^{k})),$ satisfying the Legendre condition, let us define the space 
 $$ W^{1,2}_{B,N}\left(\Omega; \Lambda^{k} \right):= \left\lbrace \omega \in W^{1,2}\left(\Omega; \Lambda^{k} \right): \nu\lrcorner \left( B(x)\omega \right) = 0 \text{ on } 
 \partial\Omega \right\rbrace . $$ Now a weak solution $\omega \in W^{1,2}_{B,N}\left(\Omega; \Lambda^{k} \right)$ for 
 \eqref{bvp hodge elliptic system general full regularity normal} satisfies  
 \begin{multline}\label{weakformulationgeneralnormal}
 \int_{\Omega}  \langle A (x)d\omega , d\phi \rangle   + \int_{\Omega}\langle \delta \left( B(x) \omega \right) , \delta \left( B(x) \phi \right) \rangle 
  + \lambda  \int_{\Omega}\langle B(x)\omega , \phi \rangle   \\
  +\int_{\Omega} \langle  f, \phi \rangle -  \int_{\Omega} \langle  F, d\phi \rangle = 0,   
\end{multline}
for all $\phi \in W^{1,2}_{B,N}\left(\Omega; \Lambda^{k} \right).$ Setting $\beta = B(x)\omega$ and $\psi = B(x)\phi,$ we see immediately that 
it is enough to prove the regularity estimates for $\beta \in  W^{1,2}_{N}\left(\Omega; \Lambda^{k} \right)$ satisfying,  
for all $\psi \in W^{1,2}_{N}\left(\Omega; \Lambda^{k} \right),$
\begin{multline}\label{weakformW12N}
 \int_{\Omega}  \langle A (x)d \left( B^{-1}(x)\beta \right), d\left( B^{-1}(x)\psi \right) \rangle   + \int_{\Omega}\langle \delta \beta  , \delta \psi \rangle 
  + \lambda  \int_{\Omega}\langle \left( B^{-1}(x)\right)^{T}\beta , \psi \rangle   \\
  +\int_{\Omega} \langle  f, \psi \rangle -  \int_{\Omega} \langle  F, d\left( B^{-1}(x)\psi \right) \rangle = 0.   
\end{multline}

\par The rest is as before. In view of remark \ref{flatteninglemmanormal}, proving the regularity estimates for $\beta$ in \eqref{weakformW12N} comes down to proving 
 the estimates for $u \in W_{N}^{1,2}(B_{R}^{+};\Lambda^{k})$ satisfying 
\begin{align*}
 \int_{B_{R}^{+}} \langle \bar{A}(d \left(  \bar{B}^{-1} u \right) ) ; d \left( \bar{B}^{-1} \psi \right) \rangle + \int_{B_{R}^{+}} \langle \delta  u  ; \delta \psi \rangle 
 +  \int_{B_{R}^{+}} \langle \mathcal{P} & ;   \psi  \rangle 
  -\int_{B_{R}^{+}} \langle \mathcal{Q} ; \nabla\psi \rangle \notag \\&+ \int_{B_{R}^{+}} \langle \mathrm{S}\nabla u ; \nabla\psi \rangle= 0,
   \end{align*}
for all $\psi \in W_{N}^{1,2}(B_{R}^{+};\Lambda^{k}),$ where $\bar{A}$ and $\bar{B}$ satisfy the Legendre condition, 
$\mathcal{P} = \widetilde{f} + \mathrm{P}u + \mathrm{R} \nabla u$  and $\mathcal{Q} = \widetilde{F}  -  \mathrm{Q}u ,$ with $\widetilde{f}, \widetilde{F}, \mathrm{P}, 
\mathrm{Q}, \mathrm{R}, \mathrm{S}$ are as in 
remark \ref{flatteninglemmanormal}, for $R>0$ suitably small. 

\par Note that by the Gaffney type 
inequality \eqref{newgaffneynormal}, for any $v \in W_{N}^{1,2}(B_{R}^{+};\Lambda^{k})$ we have, 
 \begin{align*}
  \int_{B_{R}^{+}} \langle \bar{A}(d \left( \bar{B}^{-1}v\right) ) ; d\left( \bar{B}^{-1}v \right) \rangle + \int_{B_{R}^{+}} \left\lvert \delta  v \right\rvert^{2} 
 &\geq c\left(\int_{B_{R}^{+}}  \lvert d\left( \bar{B}^{-1}v \right)\rvert^{2} + \int_{B_{R}^{+}} \lvert \delta v \rvert^{2} \right) \\&\geq c  \int_{B_{R}^{+}} \lvert \nabla v \rvert^{2} .
 \end{align*}
To check that he $L^{2}$ norm of the $D_{nn}$ derivative can be estimated from the equation, we define the linear map $\widetilde{A}:\mathbb{R}^{\tbinom{n}{k}\times n} \rightarrow \mathbb{R}^{\tbinom{n}{k}\times n}$ by the pointwise algebraic identities
$$\langle \widetilde{A} \left( a_{1} \otimes b_{1} \right) ;  a_{2} \otimes b_{2} \rangle = \langle \bar{A}\left( a_{1} \wedge \bar{B}^{-1}b_{1} \right) ; a_{2} 
\wedge \bar{B}^{-1}b_{2} \rangle 
+ \langle a_{1} \lrcorner b_{1} ; a_{2} \lrcorner b_{2}  \rangle,  $$    for every $ a_{1},a_{2} \in \Lambda^{1}$, viewed also as vectors in $\mathbb{R}^{n}$  and $b_{1}, b_{2} \in \Lambda^{k},$
viewed also as vectors in $\mathbb{R}^{\tbinom{n}{k}}$ and the maps $\widetilde{A}^{pq}:\mathbb{R}^{\tbinom{n}{k}} \rightarrow \mathbb{R}^{\tbinom{n}{k}} $ 
by \eqref{Apqdefinition}. Now for every $\xi \in \mathbb{R}^{\tbinom{n}{k}},$ we have, 
 \begin{align*}
  \langle \widetilde{A}^{nn} \xi ; \xi \rangle &= \langle \widetilde{A} (e_{n} \otimes \xi) ; e_{n} \otimes \xi \rangle 
  = \langle \bar{A}\left( e_{n} \wedge \bar{B}^{-1}\xi \right) ; e_{n} \wedge \bar{B}^{-1}\xi \rangle 
+ \lvert e_{n} \lrcorner \xi\rvert^{2} \\&\geq \gamma_{A}\lvert e_{n} \wedge \bar{B}^{-1}\xi \rvert^{2} 
+ \left\lvert e_{n} \lrcorner \xi \right\rvert^{2} 
\geq c_{2} \left\lvert \xi \right\rvert^{2}, 
 \end{align*}
 where $c_{2}>0$ is the constant given by lemma \ref{ellipticitylemma}. This proves $\widetilde{A}^{nn}$ is invertible. Thus the regularity estimates in 
 Section \ref{boundary estimates} 
 carry over to this case as well.

\subsection{Main theorems}
\subsubsection{General Hodge type systems}
\begin{theorem}\label{generalHodgesystemtheorem}
  Let $1 \leq k \leq n-1,$  $r \geq 0$ be integers and $ 0 <  \gamma < 1$ and $1 < p < \infty$ be real numbers. 
  Let $\Omega \subset \mathbb{R}^n$ be an open, bounded $C^{r+2},$ respectively $C^{r+2, \gamma},$ set. Let 
 $A \in C^{r+1}(\overline{\Omega} ; L(\Lambda^{k+1},\Lambda^{k+1}))$, respectively 
 $C^{r+1, \gamma}(\overline{\Omega} ; L(\Lambda^{k+1},\Lambda^{k+1})),$ satisfy the Legendre-Hadamard condition and $B \in C^{r+2}(\overline{\Omega} ; L(\Lambda^{k},\Lambda^{k}))$, respectively 
 $C^{r+2, \gamma}(\overline{\Omega} ; L(\Lambda^{k},\Lambda^{k})),$ satisfy the Legendre condition.
Then the following holds.
\begin{enumerate}
 \item There exists a constant $\rho \in \mathbb{R}$ and an at most countable set 
$\sigma \subset ( -\infty, \rho )$, with no limit points except possibly $- \infty.$
\item The following boundary value problem,
\begin{equation}\label{eigenvalue problem hodge system general full regularity}
 \left\lbrace \begin{gathered}
                \delta ( A (x) d\alpha ) + \left( B(x) \right)^{T}d \delta\left( B(x) \alpha \right)  = \sigma_{i} B(x) \alpha  \text{ in } \Omega, \\
                 \nu\wedge \alpha = 0 \text{  on } \partial\Omega, \\
                 \nu\wedge \delta \left( B(x) \alpha \right) = 0 \text{ on } \partial\Omega.
                \end{gathered} 
                \right. \tag{$EP_{T}$}
\end{equation}
has non-trivial solutions $\alpha \in W^{r+2,p}(\Omega, \Lambda^{k}),$ respectively $C^{r+2,\gamma}(\overline{\Omega}, \Lambda^{k}),$ if and 
only if  $\sigma_{i} \in \sigma$ and the space of solutions to \eqref{eigenvalue problem hodge system general full regularity} is finite-dimensional for 
any $\sigma_{i} \in \sigma.$
\item If $\lambda \notin \sigma$, then 
for any $f \in W^{r,p}(\Omega, \Lambda^{k}) $, respectively $C^{r,\gamma}(\overline{\Omega}, \Lambda^{k}),$ and any $\omega_{0} \in W^{r+2,p}(\Omega, \Lambda^{k}),$ 
respectively $C^{r+2,\gamma}(\overline{\Omega}, \Lambda^{k}),$ there exists a unique 
 solution $\omega \in W^{r+2,p}(\Omega, \Lambda^{k}),$ respectively $C^{r+2,\gamma}(\overline{\Omega}, \Lambda^{k}),$ to the following boundary value problem:
 \begin{equation}\label{bvp hodge elliptic system general full regularity}
 \left\lbrace \begin{gathered}
                \delta ( A (x) d\omega )  + \left( B(x) \right)^{T}d \delta\left( B(x) \omega \right)   =  \lambda B(x)\omega + f  \text{ in } \Omega, \\
                \nu\wedge \omega = \nu\wedge\omega_{0} \text{  on } \partial\Omega. \\
                \nu\wedge \delta \left( B(x)\omega \right) = \nu\wedge\delta \left( B(x) \omega_{0} \right) \text{ on } \partial\Omega, 
                \end{gathered} 
                \right. \tag{$P_{T}$}
\end{equation}
which satisfies the estimate
\begin{align*}
 \left\lVert \omega \right\rVert_{W^{r+2,p}} \leq c \left( \left\lVert \omega\right\rVert_{W^{r,p}} + \left\lVert f\right\rVert_{W^{r,p}} 
 + \left\lVert \omega_{0} \right\rVert_{W^{r+2,p}}\right),
\end{align*}
respectively,  
\begin{align*}
 \left\lVert \omega \right\rVert_{C^{r+2,\gamma}} \leq c \left( \left\lVert \omega \right\rVert_{C^{r,\gamma}} + \left\lVert f \right\rVert_{C^{r,\gamma}} 
 + \left\lVert \omega_{0} \right\rVert_{C^{r+2,\gamma}}\right).
\end{align*}
\end{enumerate}
\end{theorem}
\begin{remark}\label{remarkhodgegeneraltan}
 As in remark \ref{spectrum}, it is easy to see that if $A$ has constant coefficients or satisfies 
 the Legendre condition, then $\sigma \subset ( -\infty, 0].$ If $\mathcal{H}_{T}\left( \Omega; \Lambda^{k} \right) \neq \lbrace 0 \rbrace,$  then clearly, for any nontrivial 
 $h \in \mathcal{H}_{T}\left( \Omega; \Lambda^{k} \right),$ $\alpha = d\beta + h$ is a nontrivial solution for \eqref{eigenvalue problem hodge system general full regularity} with 
 $\sigma_{i} = 0,$ where $\beta$ is a solution of 
 \begin{align*}
 \left\lbrace \begin{aligned}\delta ( B d \beta ) &=  - \delta ( Bh ) &&\text{ in } \Omega, \\
 \delta \beta &= 0 &&\text{ in } \Omega, \\
 \nu \wedge \beta &= 0 &&\text{ on  } \partial \Omega.\end{aligned}\right. \end{align*}
Thus the eigenforms with eigenvalue $0$ are in one-to-one correspondence with the nontrivial harmonic fields in $\mathcal{H}_{T}\left( \Omega; \Lambda^{k} \right).$ Consequently,  
\eqref{bvp hodge elliptic system general full regularity} with $\lambda = 0$ can be solved for any $f$ satisfying the additional conditions $\delta f = 0$ in $\Omega$ and 
$ f \in \left( \mathcal{H}_{T}\left( \Omega; \Lambda^{k} \right)\right)^{\perp}.$ If $\mathcal{H}_{T}\left( \Omega; \Lambda^{k} \right) = \lbrace 0 \rbrace,$ no extra condition 
on $f$ is needed.    
\end{remark}

\begin{proof}
We divide the proof in two steps. \smallskip

\noindent \emph{Step 1 (Existence in $L^{2}$):} 
 For a given $\lambda \in \mathbb{R}$, the bilinear operator 
$a_{\lambda}: W_{T}^{1,2}(\Omega;\Lambda^k) \times W_{T}^{1,2}(\Omega;\Lambda^k) \rightarrow \mathbb{R}$ defined by, 
\begin{align*}
 a_{\lambda}(u,v)  =\int_{\Omega}   \langle A (x)d u  , d v \rangle + \int_{\Omega}   \langle \delta \left( B(x) u \right)  , \delta \left( B(x) v \right) \rangle 
 + \lambda \int_{\Omega} \langle B(x) u , v \rangle ,
\end{align*}
 is continuous and coercive, by theorem \ref{newgaffney} and the G\aa{}rding inequality, for any $\lambda $ large enough.  
Now Lax-Milgram theorem, together with the compact embedding of $W_{T}^{1,2}(\Omega;\Lambda^k)$ into $L^{2}(\Omega;\Lambda^k)$ implies that Fredholm theory holds.
Thus, for any $g \in L^{2}(\Omega, \Lambda^{k})$ and any $\lambda \notin \sigma,$ there exists an unique solution $\overline{\omega} \in W_{T}^{1,2}(\Omega;\Lambda^k)$ 
satisfying, for all  $\theta \in W_{T}^{1,2}(\Omega;\Lambda^k),$ 
\begin{align}\label{equationwithfgeneral}
 \int_{\Omega}  \langle A (x) d\overline{\omega} , d\theta \rangle + \langle \delta \left( B(x)\overline{\omega}\right) , \delta \left( B(x) \theta \right) \rangle + 
  \lambda  \int_{\Omega}\langle  B(x)\overline{\omega} , \theta \rangle 
  +\int_{\Omega} \langle  g, \theta\rangle = 0 . 
\end{align}

\noindent \emph{Step 2 (Regularity and boundary conditions):} Now the desired regularity estimates follow the same way as in theorem \ref{boundary Wrp regularity linear}, respectively 
theorem \ref{boundary Cralpha regularity linear} in the H\"{o}lder case, in view of the discussion in Section \ref{regularitydiscssiongeneraltan}.  The estimates for $1 < p < 2$ also extends the existence theory to the case $1 < p < 2$ by usual density arguments.   
From \eqref{equationwithfgeneral}, integrating by parts, we obtain, for all  $ \phi \in W_{T}^{1,2}(\Omega, \Lambda^{k}),$
\begin{multline*}
 \int_{\Omega} \langle \delta  ( A (x) d\overline{\omega} )    + \left( B(x) \right)^{T}d \delta\left( B(x) \overline{\omega} \right) ;  \phi \rangle 
 - \int_{\Omega} \langle  \lambda B(x)\overline{\omega} + g ; \phi \rangle \\ =  \int_{\partial\Omega} \left( \langle A(x) d\overline{\omega}; \nu\wedge\phi \rangle   
+  \langle \nu\wedge\delta \left( B(x)\overline{\omega}\right) ; \left( B(x)\phi\right) \rangle \right).
\end{multline*}
Thus taking  $\phi \in C_{c}^{\infty}(\Omega, \Lambda^{k})$ we see that $\overline{\omega}$ satisfies  
 $$\delta ( A(x) d\overline{\omega} )   + \left( B(x) \right)^{T}d \delta\left( B(x) \overline{\omega} \right) = \lambda\left( B(x)\overline{\omega} \right) + g 
 \qquad \text{ in } \Omega $$ and the integral on the boundary vanish separately. But since $\phi \in W_{T}^{1,2}(\Omega, \Lambda^{k}),$ $\nu\wedge\phi = 0.$ Hence we obtain,
 $$ \int_{\partial\Omega}   \langle \left( B(x)\right)^{T} \left[ \nu\wedge\delta \left( B(x)\overline{\omega}\right) \right] ; \phi \rangle  =0 \qquad 
 \text{ for any } \phi \in W_{T}^{1,2}(\Omega; \Lambda^{k}).$$ 
 Extending $\nu$ as a $C^{1}$ function inside $\Omega,$ we see $\nu\wedge\delta \left( B(x)\overline{\omega}\right) \in  W_{T}^{1,2}(\Omega; \Lambda^{k}).$
 Thus using the Legendre condition of $B,$ we obtain, 
 $$0 = \int_{\partial\Omega}   \langle 
 \left( B(x)\right)^{T} \left[ \nu\wedge\delta \left( B(x)\overline{\omega}\right) \right] ; \left[ \nu\wedge\delta \left( B(x)\overline{\omega}\right) \right] \rangle 
 \geq c \int_{\partial\Omega} \left\lvert \nu\wedge\delta \left( B(x)\overline{\omega}\right) \right\rvert^{2}.$$
 This proves $\nu\wedge\delta \left( B(x)\overline{\omega}\right)  = 0$ on $\partial\Omega.$ Now taking 
 $g= f + \lambda B(x) \omega_{0} -\delta (A (x) d\omega_{0}) - \left( B(x) \right)^{T}d \delta\left( B(x) \omega_{0} \right) $ and 
 setting $\omega = \overline{\omega} + \omega_{0},$ we see that $\omega $ 
    is a solution to \eqref{bvp hodge elliptic system full regularity} with the desired regularity. This finishes the proof. 
\end{proof}

\begin{theorem}\label{generalHodgesystemtheoremnormal}
 Let $1 \leq k \leq n-1,$  $r \geq 0$ be integers and $ 0 <  \gamma < 1$ and $1 < p < \infty$ be real numbers. 
 Let $\Omega \subset \mathbb{R}^n$ be an open, bounded $C^{r+2},$ respectively $C^{r+2, \gamma},$ set. Let 
 $A \in C^{r+1}(\overline{\Omega} ; L(\Lambda^{k+1},\Lambda^{k+1}))$, respectively 
 $C^{r+1, \gamma}(\overline{\Omega} ; L(\Lambda^{k+1},\Lambda^{k+1})),$ and $B \in C^{r+2}(\overline{\Omega} ; L(\Lambda^{k},\Lambda^{k}))$, respectively 
 $C^{r+2, \gamma}(\overline{\Omega} ; L(\Lambda^{k},\Lambda^{k})),$ both satisfy the Legendre condition.
Then the following holds.
\begin{enumerate}
 \item There exists a constant $\rho \in \mathbb{R}$ and an at most countable set 
$\sigma \subset ( -\infty, \rho )$, with no limit points except possibly $- \infty.$
\item The following boundary value problem,
\begin{equation}\label{eigenvalue problem hodge system general full regularity normal}
 \left\lbrace \begin{gathered}
                \delta ( A (x) d\alpha ) + \left( B(x) \right)^{T}d \delta\left( B(x) \alpha \right)  = \sigma_{i} B(x) \alpha  \text{ in } \Omega, \\
                 \nu\lrcorner \left( B(x)\alpha \right) = 0 \text{  on } \partial\Omega, \\
                 \nu\lrcorner \left( A(x) d\alpha \right) = 0 \text{ on } \partial\Omega.
                \end{gathered} 
                \right. \tag{$EP_{N}$}
\end{equation}
has non-trivial solutions $\alpha \in W^{r+2,p}(\Omega, \Lambda^{k}),$ respectively $C^{r+2,\gamma}(\overline{\Omega}, \Lambda^{k}),$ if and only if  $\sigma_{i} \in \sigma$ 
and the space of solutions to \eqref{eigenvalue problem hodge system general full regularity normal} is finite-dimensional for 
any $\sigma_{i} \in \sigma.$
\item If $\lambda \notin \sigma$, then 
for any $f \in W^{r,p}(\Omega, \Lambda^{k}) $, respectively $C^{r,\gamma}(\overline{\Omega}, \Lambda^{k}),$ and any $\omega_{0} \in W^{r+2,p}(\Omega, \Lambda^{k}),$ 
respectively $C^{r+2,\gamma}(\overline{\Omega}, \Lambda^{k}),$ there exists a unique 
 solution $\omega \in W^{r+2,p}(\Omega, \Lambda^{k}),$ respectively $C^{r+2,\gamma}(\overline{\Omega}, \Lambda^{k}),$ to the following boundary value problem:
 \begin{equation}\label{bvp hodge elliptic system general full regularity normal}
 \left\lbrace \begin{gathered}
                \delta ( A (x) d\omega )  + \left( B(x) \right)^{T}d \delta\left( B(x) \omega \right)   =  \lambda B(x)\omega + f  \text{ in } \Omega, \\
                \nu\lrcorner \left( B(x)\omega \right) = \nu\lrcorner\left( B(x)\omega_{0} \right) \text{  on } \partial\Omega. \\
                \nu\lrcorner  \left( A(x)d\omega \right) = \nu\lrcorner \left( A(x) d\omega_{0} \right) \text{ on } \partial\Omega, 
                \end{gathered} 
                \right. \tag{$P_{N}$}
\end{equation}
which satisfies the estimate
\begin{align*}
 \left\lVert \omega \right\rVert_{W^{r+2,p}} \leq c \left( \left\lVert \omega\right\rVert_{W^{r,p}} + \left\lVert f\right\rVert_{W^{r,p}} 
 + \left\lVert \omega_{0} \right\rVert_{W^{r+2,p}}\right),
\end{align*}
respectively,  
\begin{align*}
 \left\lVert \omega \right\rVert_{C^{r+2,\gamma}} \leq c \left( \left\lVert \omega \right\rVert_{C^{r,\gamma}} + \left\lVert f \right\rVert_{C^{r,\gamma}} 
 + \left\lVert \omega_{0} \right\rVert_{C^{r+2,\gamma}}\right).
\end{align*}
\end{enumerate}
\end{theorem}
\begin{remark}
 Here clearly $\sigma \subset ( -\infty, 0].$ If $\mathcal{H}_{N}\left( \Omega; \Lambda^{k} \right) \neq \lbrace 0 \rbrace,$  then for any nontrivial 
 $h \in \mathcal{H}_{N}\left( \Omega; \Lambda^{k} \right),$ $\alpha = B^{-1}\left( \delta\beta + h \right) $ is a nontrivial solution for \eqref{eigenvalue problem hodge system general full regularity normal} with 
 $\sigma_{i} = 0,$ where $\beta$ is a solution of 
 \begin{align*}
 \left\lbrace \begin{aligned} d ( B^{-1} \delta \beta ) &=  - d ( B^{-1} h ) &&\text{ in } \Omega, \\
 d\beta &= 0 &&\text{ in } \Omega, \\
 \nu \lrcorner \beta &= 0 &&\text{ on  } \partial \Omega.\end{aligned}\right. \end{align*}
 Note that the solvability of the above system is given by the dual version of Theorem \ref{maxwell linear regularity full}.
Thus the eigenforms with eigenvalue $0$ are in one-to-one correspondence with the nontrivial harmonic fields in $\mathcal{H}_{N}\left( \Omega; \Lambda^{k} \right).$ Since any such 
nontrivial $\alpha$ satisfies $d\alpha = 0$ in $\Omega,$ \eqref{bvp hodge elliptic system general full regularity normal} with $\lambda = 0$ can be solved for 
any $f$ satisfying the additional conditions $\delta f = 0$ in $\Omega,$ $\nu \lrcorner f = 0$ on $\partial\Omega$ and 
$ f \in \left( \mathcal{H}_{N}\left( \Omega; \Lambda^{k} \right)\right)^{\perp}.$ If $\mathcal{H}_{N}\left( \Omega; \Lambda^{k} \right) = \lbrace 0 \rbrace,$ no extra condition 
on $f$ is needed.    
\end{remark}
\begin{proof}
 As in the discussion in section \ref{regularitydiscssiongeneralnormal}, we start with the weak formulation 
 \eqref{weakformulationgeneralnormal} and convert it to the equivalent formulation \eqref{weakformW12N}. 
 Existence and regularity results follow in an analogous way from this after flattening the boundary (see remark \ref{flatteninglemmanormal}). The only thing that is different is the verification of the boundary condition. As before, the integral on the 
 boundary vanish separately and we obtain 
 $$ \int_{\partial\Omega}\langle \left( B^{-1}(x) \right)^{T}\left[ \nu \lrcorner A(x)d\bar{\omega} \right]; \phi \rangle = 0 \qquad \text{ for any } \phi \in W^{1,2}_{N}. $$ Plugging in 
 $\nu \lrcorner A(x)d\bar{\omega} $ in place of $\phi,$ we obtain
 $$ 0 =  \int_{\partial\Omega}\langle \left( B^{-1}(x) \right)^{T}\left[ \nu \lrcorner A(x)d\bar{\omega} \right]; \nu \lrcorner A(x)d\bar{\omega} \rangle \geq 
 c \int_{\partial\Omega} \left\lvert \nu \lrcorner A(x)d\bar{\omega} \right\rvert^{2}.$$
\end{proof}

\subsubsection{General Maxwell operators} 
\begin{theorem}
 Let $1 \leq k \leq n-1,$ $r \geq 0$ be integers and $ 0 < \gamma < 1$ and $1 < p < \infty$ be real numbers. 
 Let $\Omega \subset \mathbb{R}^n$ be an open, bounded $C^{r+2},$ respectively $C^{r+2, \gamma},$ set.
  Let $A \in C^{r+1}(\overline{\Omega} ; L(\Lambda^{k+1},\Lambda^{k+1}))$, respectively 
 $C^{r+1, \gamma}(\overline{\Omega} ; L(\Lambda^{k+1},\Lambda^{k+1})),$ and $B \in C^{r+2}(\overline{\Omega} ; L(\Lambda^{k},\Lambda^{k}))$, respectively 
 $C^{r+2, \gamma}(\overline{\Omega} ; L(\Lambda^{k},\Lambda^{k})),$ satisfy the Legendre condition. 
Let  $\omega_{0} \in W^{r+2,p}(\Omega, \Lambda^{k})$, respectively $C^{r+2,\gamma}(\overline{\Omega} ; \Lambda^{k}),$ 
 $f \in W^{r,p}(\Omega, \Lambda^{k+1}),$ respectively $C^{r,\gamma}(\overline{\Omega} ; \Lambda^{k+1})$ and $g \in W^{r+1,p}(\Omega, \Lambda^{k-1}),$
  respectively $C^{r+1,\gamma}(\overline{\Omega} ; \Lambda^{k-1}),$ and $\lambda \geq 0.$ Suppose $f,$ $g$ and $\lambda$  satisfy 
  \begin{align}
   \delta f + \lambda g = 0 \text{ and } \delta g = 0 \text{ in } \Omega . \tag{C}
  \end{align}
\begin{itemize}
 \item[(i)] Suppose $ g \in \left(\mathcal{H}_T(\Omega;\Lambda^{k-1})\right)^{\perp}$ and if $\lambda = 0,$ assume in addition that 
 $f \in \left(\mathcal{H}_T(\Omega;\Lambda^{k})\right)^{\perp}.$  Then   
 the following boundary value problem, 
 \begin{equation} \label{problemMaxwellgeneraltan}
   \left\lbrace \begin{gathered}
                \delta ( A (x) d\omega )   = \lambda B(x)\omega + f  \text{ in } \Omega, \\
                \delta \left( B(x) \omega \right) = g \text{ in } \Omega, \\
                \nu\wedge \omega = \nu\wedge \omega_{0} \text{  on } \partial\Omega,
                \end{gathered} 
                \right. \tag{$PM_{T}$}
\end{equation}
has a unique solution $\omega \in W^{ r+2, p}(\Omega, \Lambda^{k}),$ 
 respectively $C^{r+2,\gamma}(\overline{\Omega} ; \Lambda^{k}),$ 
satisfying the estimates 
\begin{align*}
 \left\lVert \omega \right\rVert_{W^{r+2,p}} \leq c \left(  \left\lVert  \omega \right\rVert_{L^{p}} + \left\lVert f\right\rVert_{W^{r,p}}  
 + \left\lVert g\right\rVert_{W^{r,p}} + \left\lVert \omega_{0}\right\rVert_{W^{r+2,p}}\right), 
\end{align*}
respectively,  
\begin{align*}
 \left\lVert \omega \right\rVert_{C^{r+2,\gamma}} \leq c \left( \left\lVert  \omega \right\rVert_{C^{0,\gamma}} 
 + \left\lVert f \right\rVert_{C^{r,\gamma}} + \left\lVert g \right\rVert_{C^{r,\gamma}} 
 + \left\lVert \omega_{0} \right\rVert_{C^{r+2,\gamma}} \right).
\end{align*}

\item[(ii)]  Suppose \begin{align*}
 \nu\lrcorner g = \nu \lrcorner \delta \left(  B(x)\omega_{0} \right) \quad\text{ and }\quad \nu\lrcorner f 
 = \nu \lrcorner \left[ \delta \left( A(x)d\omega_{0} \right) - \lambda 
 B(x)\omega_{0} \right] \quad\text{ on } \partial\Omega
\end{align*}
and 
$$ \int_{\Omega} \left\langle g ; \psi \right\rangle  - \int_{\partial\Omega} \left\langle \nu\lrcorner \left( B(x)\omega_{0}\right); \psi \right\rangle = 0 \qquad \text{ for all } \psi 
\in \mathcal{H}_N(\Omega;\Lambda^{k-1}).$$
 If $\lambda =0,$ assume in addition that  
$$\int_{\Omega} \left\langle f ; \phi \right\rangle  - \int_{\partial\Omega} \left\langle \nu\lrcorner \left( A(x)d\omega_{0} \right); \phi \right\rangle = 0 \qquad \text{ for all } \phi 
\in \mathcal{H}_N(\Omega;\Lambda^{k}). $$ Then the following boundary value problem, 
 \begin{equation} \label{problemMaxwellgeneralnormal}
   \left\lbrace \begin{gathered}
                \delta ( A (x) d\omega )   = \lambda B(x)\omega + f  \text{ in } \Omega, \\
                \delta \left( B(x) \omega \right) = g \text{ in } \Omega, \\
                \nu\lrcorner \left( B(x) \omega \right) = \nu\lrcorner \left( B(x) \omega_{0} \right)  \text{  on } \partial\Omega, \\
                \nu\lrcorner \left( A(x) d\omega \right) = \nu\lrcorner \left( A(x) d\omega_{0} \right)  \text{  on } \partial\Omega. 
                \end{gathered} 
                \right. \tag{$PM_{N}$}
\end{equation}
has a unique solution $\omega \in W^{ r+2, p}(\Omega, \Lambda^{k}),$ respectively $C^{r+2,\gamma}(\overline{\Omega} ; \Lambda^{k}),$  
satisfying the estimates 
\begin{align*}
 \left\lVert \omega \right\rVert_{W^{r+2,p}} \leq c \left(  \left\lVert  \omega \right\rVert_{L^{p}} + \left\lVert f\right\rVert_{W^{r,p}}  
 + \left\lVert g\right\rVert_{W^{r,p}} + \left\lVert \omega_{0}\right\rVert_{W^{r+2,p}}\right), 
\end{align*}
respectively,  
\begin{align*}
 \left\lVert \omega \right\rVert_{C^{r+2,\gamma}} \leq c \left( \left\lVert  \omega \right\rVert_{C^{0,\gamma}} 
 + \left\lVert f \right\rVert_{C^{r,\gamma}} + \left\lVert g \right\rVert_{C^{r,\gamma}} 
 + \left\lVert \omega_{0} \right\rVert_{C^{r+2,\gamma}} \right).
\end{align*}
\end{itemize}
\end{theorem}
 \begin{proof}
  We show only part(i) in the Sobolev case, the other cases are similar. 
 We first find $G \in W^{r+2,p}(\Omega, \Lambda^{k})$ such that 
  \begin{equation*} 
   \left\lbrace \begin{aligned}
                dG = 0  \quad \text{and} \quad &  \delta (B(x)G) = g - \delta\left( B(x)\omega_{0}\right) &&\text{ in } \Omega, \\
                \nu\wedge G &= 0 &&\text{  on } \partial\Omega.
                \end{aligned} 
                \right. 
                \end{equation*}
Now we find $u \in W^{r+2,p}(\Omega, \Lambda^{k})$ solving 
\begin{equation}\label{uequation}
 \left\lbrace \begin{aligned}
                \delta ( A (x) du )  + \left( B(x) \right)^{T}d \delta\left( B(x) u \right)   &=  \lambda B(x)u + \widetilde{f}  &&\text{ in } \Omega, \\
                \nu\wedge u &= 0 &&\text{  on } \partial\Omega. \\
                \nu\wedge \delta \left( B(x)u\right) &= 0  &&\text{ on } \partial\Omega. 
                \end{aligned} 
                \right. 
\end{equation}
where 
$\widetilde{f} = f + \lambda B(x)\omega_{0} + \lambda B(x)G - \delta ( A (x) d\omega_{0} ).$ Note that if $\lambda = 0,$ 
then $\widetilde{f} \in \left(\mathcal{H}_T(\Omega;\Lambda^{k})\right)^{\perp}$ and 
$\delta\widetilde{f} = 0$ and thus \eqref{uequation} can always be solved for any $\lambda \geq 0 $ (see remark \ref{remarkhodgegeneraltan}).
This implies $v= \delta \left( B(x)u \right)$ is a $W^{1,p}$ weak solution of 
\begin{equation*}
 \left\lbrace \begin{aligned}
                \delta \left( B^{T}(x) dv  \right)  &=  \lambda v   &&\text{ in } \Omega, \\
                \delta v &= 0 &&\text{  in } \Omega. \\
                \nu\wedge v &= 0  &&\text{ on } \partial\Omega. 
                \end{aligned} 
                \right. 
\end{equation*}
But since no non-zero harmonic field can be co-exact, we have  
$ \delta \left( B(x)u \right) = 0$ in $\Omega.$
Now it is easy to check that  $\omega = \omega_{0} + u + G$ solves \eqref{problemMaxwellgeneraltan}.  
\end{proof}

\section{Appendix: Flattening the boundary}
Now we prove a lemma to transfer our problem to a half-ball by localizing and flattening the boundary near a boundary point. 
This can be done in such a way that we get a constant coefficient system of the same form up to a small perturbations and lower order terms.   

\begin{lemma}\label{flattening and freezing coefficients}
 Let $r\geq 0$ is an integer and $0 < \gamma < 1.$ Let $\partial\Omega$ is of class $C^{r+2}$, resp. $C^{r+2,\gamma}.$  
Let $\omega \in W_{T}^{1,2}(\Omega, \Lambda^{k})$ satisfy, for all $ \phi \in W_{T}^{1,2}(\Omega; \Lambda^{k}),$
 \begin{multline}\label{bilinear form w22 regularity}
 \int_{\Omega}  \langle A (x)d\omega , d\phi \rangle   + \int_{\Omega}\langle \delta \left( B(x)\omega \right) , \delta \left( B(x)\phi \right) \rangle 
  + \lambda  \int_{\Omega}\langle  B(x)\omega , \phi \rangle \\
  +\int_{\Omega} \langle  f, \phi \rangle  -\int_{\Omega} \langle  F, d\phi \rangle = 0.   
\end{multline}
\noindent Suppose  $f \in W^{r,p}(\Omega, \Lambda^{k}),$ resp. $C^{r,\gamma}(\overline{\Omega}, \Lambda^{k}),$ $ F \in W^{r+1,p}(\Omega, \Lambda^{k+1}),$ resp. 
 $ C^{r+1,\gamma}(\overline{\Omega}, \Lambda^{k+1})$ and $\lambda \in \mathbb{R}.$\smallskip 
 
\noindent Let $A \in C^{r+1}\left(\overline{\Omega}; L(\Lambda^{k+1},\Lambda^{k+1})\right),$ resp.
 $C^{r+1,\gamma}\left(\overline{\Omega}; L(\Lambda^{k+1},\Lambda^{k+1})\right),$ satisfy the Legendre-Hadamard condition,
  and let $B \in C^{r+2}\left(\overline{\Omega}; L(\Lambda^{k},\Lambda^{k})\right),$ resp. 
 $C^{r+2,\gamma}\left(\overline{\Omega}; L(\Lambda^{k},\Lambda^{k})\right),$ satisfies the Legendre condition. \smallskip

\noindent Then for every given $R > 0$ and for every $x_0 \in \partial\Omega,$  there exist 
 
 \begin{itemize}
 \item[(i)] a constant matrix $\bar{A}:\Lambda^{k+1} \rightarrow \Lambda^{k+1}$ satisfying the Legendre-Hadamard condition and a constant 
 matrix $\bar{B}:\Lambda^{k} \rightarrow \Lambda^{k}$ satisfying the Legendre condition, 
 \item[(ii)] matrix functions $ \mathrm{P} ,\mathrm{Q} ,\mathrm{R}$  and  $\mathrm{S} $
 which are $C^{r+1},$ respectively $C^{r+1, \gamma},$ in $B_{R}^{+}$ 
 with
 \begin{equation*}
 \mathcal{C}_{\mathrm{S}}\left( R \right) \leq c_{0}  \left( \mathcal{C}_{A}\left( R \right) + \mathcal{C}_{B}\left( R \right) + \mathcal{C}_{\nabla B}\left( R \right)+ \mathcal{C}_{\Phi}\left( R \right)\right) , 
\end{equation*}
depending only on $A$, $B$, $\Phi$, $\theta$, $W$ and $R$,   where $c_{0} > 0$ is a constant independent of $R$ and $\mathcal{C}_{\mathrm{S}},$ $\mathcal{C}_{A}, 
\mathcal{C}_{B}, \mathcal{C}_{\nabla B}$ and $ \mathcal{C}_{\Phi} $ 
 are the modulus of continuity of $\mathrm{S},$ $A$, $B$, $\nabla B $ and $\Phi,$ respectively,
 
\item[(iii)] a neighborhood $W$ of $x_0$ in $\mathbb{R}^n ,$  a function $\theta \in C_{c}^{\infty}\left( W \right) $, an admissible boundary coordinate system 
  $\Phi \in \operatorname*{Diff}^{r+2} (\overline{B_{R}};\overline{W}) $, 
 resp. $\operatorname*{Diff}^{r+2,\gamma} (\overline{B_{R}};\overline{W}) ,$ such that 
 $\Phi(0)= x_{0},$ $\Phi (B_{R}^{+} ) =\Omega \cap W$ and $\Phi (\Gamma_{R}) = \partial\Omega \cap W,$

\item[(iv)] a form $\widetilde{f} \in W^{r,p}(B_{R}^{+}; \Lambda^{k}),$ resp. $C^{r,\gamma}(\overline{B_{R}^{+}}; \Lambda^{k}),$ and a matrix field 
$\widetilde{F} \in W^{r+1,p}\left(B_{R}^{+}; \mathbb{R}^{\tbinom{n}{k}\times n}\right),$   
resp. $C^{r+1,\gamma}\left(\overline{B_{R}^{+}}; \mathbb{R}^{\tbinom{n}{k}\times n}\right),$   
with estimates on the norms by the norms of $f$ and $F$, the constants in the estimates depending only on $\Phi$, $\theta$, $W$ and $R$,
 \end{itemize}
such that 
 $u = \Phi^{\ast}(\theta\omega) \in W_{T, flat}^{1,2}(B_{R}^{+} ; \Lambda^{k}) $  satisfies, 
 for all $\psi \in W_{T, flat}^{1,2}(B_{R}^{+} ; \Lambda^{k}),$
 \begin{align*}
 \int_{B_{R}^{+}} \langle \bar{A}(du) ; d\psi \rangle + \int_{B_{R}^{+}} &\langle \delta ( \bar{B} u ) ; \delta ( \bar{B} \psi )  \rangle 
 +  \int_{B_{R}^{+}} \langle \widetilde{f} + \mathrm{P}u + \mathrm{R}\nabla u ; \psi \rangle 
   \notag \\ & - \int_{B_{R}^{+}} \langle \widetilde{F} -\mathrm{Q} u ; \nabla\psi \rangle   + \int_{B_{R}^{+}} \langle \mathrm{S} \nabla u , \nabla\psi \rangle = 0.
 \end{align*}
\end{lemma}

\begin{proof}
Given $R> 0,$ the existence of an admissible boundary co-ordinate system 
$\Phi \in \operatorname*{Diff}^{r+2} (\overline{B_{R}};\overline{W}) ,$ respectively,  
$\operatorname*{Diff}^{r+2,\gamma} (\overline{B_{R}};\overline{W})$ is standard (see \cite{Morrey1966}, also \cite{Csatothesis}). Also we can 
assume $x_{0} = 0$ and $D\Phi(0) \in  { \rm SO(n)}.$ Set $T = \left( D\Phi(0)\right)^{-1}$ and define $$\bar{A} = \left( T^{-1} \right)^{\ast} \circ A(x_{0}) \circ  \left( T \right)^{\ast} \text{ and } 
\bar{B} = \left( T^{-1} \right)^{\ast} \circ B(x_{0}) \circ  \left( T \right)^{\ast}.$$
Then $\bar{A}$ satisfies a Legendre-Hadamard condition.
Indeed, for any $a \in \Lambda^{1}, b \in \Lambda^{k},$ we have
\begin{align*}
 \langle \bar{A} (a \wedge b ); a \wedge b \rangle &= \langle \left( T^{-1} \right)^{\ast} \circ A(x_{0}) \circ  \left( T \right)^{\ast} (a \wedge b ); a \wedge b \rangle \\
&= \left( T^{-1} \right)^{\ast} \left( \langle A(x_{0}) \circ  \left( T \right)^{\ast} (a \wedge b );   \left( T \right)^{\ast} (a \wedge b ) 
\rangle \right) \\&= \left( T^{-1} \right)^{\ast} \left( \langle A(x_{0})  \left( T^{\ast}a \wedge T^{\ast} b \right);   \left( T^{\ast}a \wedge T^{\ast}b \right) 
\rangle \right) \geq \gamma_{\bar{A}} \lvert a \wedge b \rvert^{2},
\end{align*}
for some $\gamma_{\bar{A}} > 0,$ since $\left( T^{-1} \right)^{\ast},$ $T^{\ast}$ are both bijective and $ A(x_{0})$ satisfies the Legendre-Hadamard condition. 
Clearly $\bar{B}$ satisfies a Legendre condition. Choosing $\phi = (\Phi^{-1})^{\ast}\psi$ and substituting $\theta\omega = (\Phi^{-1})^{\ast} u $ in the equation satisfied by 
$\theta\omega,$ the conclusion of the lemma follows by
 applying change of variables and grouping the terms. The terms containing $f$ and $F,$  are grouped according as whether they are multiplied with $\psi$ or derivatives of 
$\psi ,$ as $\widetilde{f}$ and $\widetilde{F},$ respectively. The claimed regularity of the coefficients follow easily. 
 On the other hand, all the components of $\mathrm{S}$  contain terms which depend on the components of the 
differences  $ A(x) - A(x_{0}),$ $ B(x) - B(x_{0}),$ $\nabla B(x) - \nabla B (x_{0})$ and $D\Phi(y) - D\Phi(0),$ implying $\mathrm{S}(0) =0$ and the claimed scaling of 
$\mathcal{C}_{\mathrm{S}}.$  
\end{proof}
\begin{remark}\label{remark for the flattening lemma} Note that if $A$ satisfies the Legendre condition, then $\bar{A}$ satisfies the Legendre condition as well.
\end{remark}
\begin{remark}\label{flatteninglemmanormal} If $\omega \in W_{N}^{1,2}(\Omega, \Lambda^{k})$ satisfy, 
 \begin{multline}
 \int_{\Omega}  \langle A (x)d \left( B^{-1}(x) \omega \right) , d \left( B^{-1}(x) \phi \right) \rangle   + \int_{\Omega}\langle \delta \omega  , \delta \phi \rangle 
  + \lambda  \int_{\Omega}\langle  \omega , B^{-1}(x)\phi \rangle \\
  +\int_{\Omega} \langle  f, B^{-1}(x)\phi \rangle  -\int_{\Omega} \langle  F, d \left( B^{-1}(x)\phi \right)\rangle = 0,   
\end{multline}
for all $ \phi \in W_{N}^{1,2}(\Omega; \Lambda^{k})$ and $A$ satisfies the Legendre condition, then analogous results hold, giving the existence $W,\theta, \Phi$ and constant matrices  $\bar{A}$ and $\bar{B}$, both satisfying 
the Legendre condition such that $u = \Phi^{\ast}(\theta\omega) \in W_{N, flat}^{1,2}(B_{R}^{+} ; \Lambda^{k}) $  satisfies, 
 for all $\psi \in W_{N, flat}^{1,2}(B_{R}^{+} ; \Lambda^{k}),$
 \begin{align*}
 \int_{B_{R}^{+}} \langle \bar{A}(d \left( \bar{B}^{-1}u \right)) ; d\left( \bar{B}^{-1}\psi \right)\rangle + &\int_{B_{R}^{+}} \langle \delta  u ; \delta \psi   \rangle 
 +  \int_{B_{R}^{+}} \langle \widetilde{f} + \mathrm{P}u + \mathrm{R}\nabla u ; \psi \rangle 
   \notag \\ & - \int_{B_{R}^{+}} \langle \widetilde{F} -\mathrm{Q} u ; \nabla\psi \rangle   + \int_{B_{R}^{+}} \langle \mathrm{S} \nabla u , \nabla\psi \rangle = 0,
 \end{align*}
 with the same conclusions for $\mathrm{P},\mathrm{Q},\mathrm{R},\mathrm{S},\widetilde{f}$ and $\widetilde{F}.$
\end{remark}

\begin{acknowledgement}
The author thanks Bernard Dacorogna, Jan Kristensen and Ho\`{a}i-Minh Nguy\^{e}n for helpful 
comments and discussions. This work was conceived as a part of author's doctoral thesis in EPFL, whose support and facilities are also gratefully acknowledged. The author thanks the anonymous referee for the comments and suggestions. 
\end{acknowledgement}


\begin{thebibliography}{10}

\bibitem{AlbertiCapdeboscqMaxwell}
{\sc Alberti, G.~S., and Capdeboscq, Y.}
\newblock Elliptic regularity theory applied to time harmonic anisotropic
  {M}axwell's equations with less than {L}ipschitz complex coefficients.
\newblock {\em SIAM J. Math. Anal. 46}, 1 (2014), 998--1016.

\bibitem{BeiraoBerselliNavierStokesstressfreebc}
{\sc Beir\~ao~da Veiga, H., and Berselli, L.~C.}
\newblock Navier-{S}tokes equations: {G}reen's matrices, vorticity direction,
  and regularity up to the boundary.
\newblock {\em J. Differential Equations 246}, 2 (2009), 597--628.

\bibitem{CampanatoEllipticsystem}
{\sc Campanato, S.}
\newblock {\em Sistemi ellittici in forma divergenza. {R}egolarit\`a
  all'interno}.
\newblock Quaderni. [Publications]. Scuola Normale Superiore Pisa, Pisa, 1980.

\bibitem{ConcaPironneauNavierStokes}
{\sc Conca, C., Par\'es, C., Pironneau, O., and Thiriet, M.}
\newblock Navier-{S}tokes equations with imposed pressure and velocity fluxes.
\newblock {\em Internat. J. Numer. Methods Fluids 20}, 4 (1995), 267--287.

\bibitem{Csatothesis}
{\sc Csat{\'o}, G.}
\newblock {Some Boundary Value Problems Involving Differential Forms, PhD
  Thesis}.
\newblock {\em EPFL}, Thesis No. 5414 (2012).

\bibitem{CsatoDacGaffney}
{\sc Csat\'o, G., and Dacorogna, B.}
\newblock An identity involving exterior derivatives and applications to
  {G}affney inequality.
\newblock {\em Discrete Contin. Dyn. Syst. Ser. S 5}, 3 (2012), 531--544.

\bibitem{CsatoDacKneuss}
{\sc Csat{\'o}, G., Dacorogna, B., and Kneuss, O.}
\newblock {\em {The pullback equation for differential forms}}.
\newblock {Progress in Nonlinear Differential Equations and their Applications,
  83}. Birkh{\"a}user/Springer, New York, 2012.

\bibitem{DacGanboKneussSymplectic}
{\sc Dacorogna, B., Gangbo, W., and Kneuss, O.}
\newblock Optimal transport of closed differential forms for convex costs.
\newblock {\em C. R. Math. Acad. Sci. Paris 353}, 12 (2015), 1099--1104.

\bibitem{DuboisStokes}
{\sc Dubois, F.}
\newblock Vorticity-velocity-pressure formulation for the {S}tokes problem.
\newblock {\em Math. Methods Appl. Sci. 25}, 13 (2002), 1091--1119.

\bibitem{FriedrichsGaffney}
{\sc Friedrichs, K.~O.}
\newblock Differential forms on {R}iemannian manifolds.
\newblock {\em Comm. Pure Appl. Math. 8\/} (1955), 551--590.

\bibitem{GaffneyHarmonicoperator}
{\sc Gaffney, M.~P.}
\newblock The harmonic operator for exterior differential forms.
\newblock {\em Proc. Nat. Acad. Sci. U. S. A. 37\/} (1951), 48--50.

\bibitem{GaffneyHarmonicintegrals}
{\sc Gaffney, M.~P.}
\newblock Hilbert space methods in the theory of harmonic integrals.
\newblock {\em Trans. Amer. Math. Soc. 78\/} (1955), 426--444.

\bibitem{giaquinta-martinazzi-regularity}
{\sc Giaquinta, M., and Martinazzi, L.}
\newblock {\em {An introduction to the regularity theory for elliptic systems,
  harmonic maps and minimal graphs}}, second~ed., vol.~11 of {\em {Appunti.
  Scuola Normale Superiore di Pisa (Nuova Serie) [Lecture Notes. Scuola Normale
  Superiore di Pisa (New Series)]}}.
\newblock Edizioni della Normale, Pisa, 2012.

\bibitem{GiaquintaModicaStokes}
{\sc Giaquinta, M., and Modica, G.}
\newblock Nonlinear systems of the type of the stationary {N}avier-{S}tokes
  system.
\newblock {\em J. Reine Angew. Math. 330\/} (1982), 173--214.

\bibitem{leisMaxwellanisotropic}
{\sc Leis, R.}
\newblock Zur {T}heorie elektromagnetischer {S}chwingungen in anisotropen
  inhomogenen {M}edien.
\newblock {\em Math. Z. 106\/} (1968), 213--224.

\bibitem{MorreyHarmonic2}
{\sc Morrey, Jr., C.~B.}
\newblock A variational method in the theory of harmonic integrals. {II}.
\newblock {\em Amer. J. Math. 78\/} (1956), 137--170.

\bibitem{Morrey1966}
{\sc Morrey, Jr., C.~B.}
\newblock {\em {Multiple integrals in the calculus of variations}}.
\newblock {Die Grundlehren der mathematischen Wissenschaften, Band 130}.
  Springer-Verlag New York, Inc., New York, 1966.

\bibitem{PicardMaxwellcompactembedding}
{\sc Picard, R.}
\newblock An elementary proof for a compact imbedding result in generalized
  electromagnetic theory.
\newblock {\em Math. Z. 187}, 2 (1984), 151--164.

\bibitem{SaranenDivcurl}
{\sc Saranen, J.}
\newblock On generalized harmonic fields in domains with anisotropic
  nonhomogeneous media.
\newblock {\em J. Math. Anal. Appl. 88}, 1 (1982), 104--115.

\bibitem{SchwarzHodge}
{\sc Schwarz, G.}
\newblock {\em {Hodge decomposition---a method for solving boundary value
  problems}}, vol.~1607 of {\em {Lecture Notes in Mathematics}}.
\newblock Springer-Verlag, Berlin, 1995.

\bibitem{silthesis}
{\sc Sil, S.}
\newblock {Calculus of {V}ariations for {D}ifferential {F}orms, PhD Thesis}.
\newblock {\em EPFL}, Thesis No. 7060 (2016).

\bibitem{StampacchiaInterpolation}
{\sc Stampacchia, G.}
\newblock The spaces {${\cal L}^{(p,\lambda )}$}, {$N^{(p,\lambda )}$} and
  interpolation.
\newblock {\em Ann. Scuola Norm. Sup. Pisa (3) 19\/} (1965), 443--462.

\bibitem{weber-regularitymaxwell}
{\sc Weber, C.}
\newblock {Regularity theorems for {M}axwell's equations}.
\newblock {\em Math. Methods Appl. Sci. 3}, 4 (1981), 523--536.

\bibitem{WeckMaxwell}
{\sc Weck, N.}
\newblock Maxwell's boundary value problem on {R}iemannian manifolds with
  nonsmooth boundaries.
\newblock {\em J. Math. Anal. Appl. 46\/} (1974), 410--437.

\end{thebibliography}
\end{document}